\documentclass[11pt]{amsart}

\usepackage{epigamath}

\usepackage[english]{babel}

\numberwithin{equation}{section}

\usepackage[shortlabels]{enumitem}
\setlist[enumerate,1]{label={\rm(\arabic*)}, ref={\rm\arabic*}} 

\usepackage{amsmath,amssymb,amsthm,amsfonts,mathrsfs}
\usepackage{enumitem,graphics,xcolor,yfonts}
\usepackage[all,cmtip]{xy}
\usepackage{tikz-cd}

\usepackage{colonequals}

\newtheorem*{thm*}{Theorem}
\newtheorem*{prop*}{Proposition}
\newtheorem*{cor*}{Corollary}
\newtheorem{thm}{Theorem}[section]
\newtheorem{prop}[thm]{Proposition}
\newtheorem{cor}[thm]{Corollary}
\newtheorem{lemma}[thm]{Lemma}

\theoremstyle{definition}

\theoremstyle{remark}

\newtheorem{rmk}[thm]{Remark}

\newcommand{\bZ}{\mathbb{Z}}    
\newcommand{\bQ}{\mathbb{Q}}    
\newcommand{\bR}{\mathbb{R}}    
\newcommand{\bC}{\mathbb{C}}    

\newcommand{\cF}{\mathcal{F}}   
\newcommand{\cM}{\mathcal{M}}   
\newcommand{\cO}{\mathcal{O}}   
\newcommand{\bP}{\mathbb{P}}    
\newcommand{\cP}{\mathcal{P}}   

\newcommand{\Aut}{\operatorname{Aut}}   
\newcommand{\Hom}{\operatorname{Hom}}   
\newcommand{\uO}{\mathrm{O}}            

\newcommand{\irr}{\operatorname{irr}}          

\newcommand{\longhookrightarrow}{\lhook\joinrel\longrightarrow}

\EpigaVolumeYear{9}{2025} \EpigaArticleNr{13} \ReceivedOn{February 2, 2024}
\InFinalFormOn{November 5, 2024}
\AcceptedOn{November 21, 2024}

\title{On the irrationality of moduli spaces of \\projective hyperk\"ahler manifolds}
\titlemark{On the irrationality of moduli spaces of projective hyperk\"ahler manifolds}

\author{Daniele Agostini}
\address{Eberhard Karls Universit\"at T\"ubingen, Fachbereich Mathematik,
  Auf der Morgenstelle 10, 72072 T\"ubingen, Germany}
\email{daniele.agostini@uni-tuebingen.de}

\author{Ignacio Barros}
\address{Department of Mathematics, Universiteit Antwerpen, Middelheimlaan 1, 2020 Antwerpen, Belgium}
\email{ignacio.barros@uantwerpen.be}

\author{Kuan-Wen Lai}
\address{Department of Smart Computing and Applied Mathematics, 
Tunghai University, 
No.~1727, Sec.~4, Taiwan~Blvd., Xitun~Dist., 
Taichung~City 407224, Taiwan\\
National Center for Theoretical Science,
No.~1, Sec.~4, Roosevelt~Rd., Taipei~City 106319, Taiwan}
\email{kwlai@thu.edu.tw}

\authormark{D.~Agostini, I.~Barros, and K.-W.~Lai}

\AbstractInEnglish{The aim of this paper is to estimate the irrationality of moduli spaces of hyperk\"ahler manifolds of types K3$^{[n]}$, Kum$_{n}$, OG6, and OG10. We prove that the degrees of irrationality of these moduli spaces are bounded from above by a universal polynomial in the dimension and degree of the manifolds they parametrize. We also give a polynomial bound for the degrees of irrationality of moduli spaces of $(1,d)$-polarized abelian surfaces.}


\acknowledgement{D.A. is supported by the Deutsche Forschungsgemeinschaft (DFG, German Research Foundation) Project 530132094. I.B. was supported by the Deutsche Forschungsgemeinschaft (DFG, German Research Foundation) -- SFB-TRR 358/1 2023 -- 491392403 and by the Fonds voor Wetenschappelijk Onderzoek – Vlaanderen (FWO, Research Foundation – Flanders) -- G0D9323N. K.-W. L. was supported by the ERC Synergy Grant HyperK (ID: 854361) and is currently supported by the NSTC Research Grant (113-2115-M-029-003-MY3).}

\begin{document}



\maketitle

\begin{prelims}

\DisplayAbstractInEnglish

\bigskip

\DisplayKeyWords

\medskip

\DisplayMSCclass

\end{prelims}


\newpage

\setcounter{tocdepth}{1}

\tableofcontents


\section{Introduction}
\label{sect:intro}

Perhaps the coarsest invariant measuring birational complexity is the Kodaira dimension. The computation of this invariant in the moduli context has been a guiding question in the past decades. A much finer, but also harder to compute, collection of invariants measuring birational complexity go by the name of \emph{measures of irrationality}. One of them, called the \emph{degree of irrationality}, is defined for a variety $X$ as the minimal possible degree of a dominant rational map
$
    X\dashrightarrow\bP^{\dim(X)}.
$
This invariant, denoted as $\irr(X)$, was first introduced in \cite{MH82} and received revived attention after \cite{BPELU17}. Notice that $\operatorname{irr}\left(X\right)=1$ if and only if $X$ is rational. In this sense, $\operatorname{irr}\left(X\right)$ measures how far $X$ is from being rational. Deciding whether a variety~$X$ is rational is a famously hard problem in algebraic geometry, suggesting that the first approach to study $\irr(X)$ is to find bounds.

In the moduli context, Donagi proposed to find bounds on measures of irrationality for classical moduli spaces such as those of curves $\mathcal{M}_g$ and principally polarized abelian varieties $\mathcal{A}_g$; see \cite[Problem~4.4]{BPELU17}. These spaces are of general type when $g$ is large enough, see \cite{Tai82,HM82,Mum83}, so their degrees of irrationality are at least $2$ for large $g$. To the best of our knowledge, there is no known upper bound on their degrees of irrationality.


In this paper, we continue our study on the irrationality of various modular varieties initiated in \cite{ABL23}. Our main objects of study are moduli spaces of $(1,d)$-polarized abelian surfaces $\mathcal{A}_{(1,d)}$ and moduli spaces of projective hyperk\"ahler manifolds $\mathcal{M}_{\Lambda,2d}^\gamma$ of known deformation types. As is the case for moduli spaces of curves and principally polarized abelian varieties, components of $\mathcal{A}_{(1,d)}$ and $\mathcal{M}_{\Lambda,2d}^\gamma$ become of general type when certain invariants grow; see \cite{OGr89,GH96,San97,Erd04} for the former and \cite{GHS07,GHS10,GHS11,Ma18,BBBF23} for the latter. Our first main result gives bounds for degrees of irrationality of $\mathcal{A}_{(1,d)}$.

\begin{thm}
For any $\varepsilon>0$, there exists a constant $C_{\varepsilon}>0$ independent of $d$ such that
\[
\operatorname{irr}\left(\mathcal{A}_{(1,d)}\right)\leq C_{\varepsilon}\cdot d^{8+\varepsilon}.
\]
Further, for special series of $d$, we have the following bounds:
\begin{itemize}
\item If $d$ is square-free, then 
\[
\operatorname{irr}\left(\mathcal{A}_{(1,d)}\right)\leq C_{\varepsilon}\cdot d^{4+\varepsilon}.
\]
\item If $d$ is a perfect square, then 
\[
\operatorname{irr}\left(\mathcal{A}_{(1,d)}\right)\leq C_{\varepsilon}\cdot d^{2+\varepsilon}.
\]
\item Fix $a,b,c\in\mathbb{Z}$ which satisfy $4ac-b^2 < 0$. Suppose that $d$ is square-free and that it is of the form $d=aX^2-bXY+cY^2$. Then for any $\varepsilon>0$, there exists a constant $C_{\varepsilon} = C_{\varepsilon}(a,b,c) > 0$ independent of $d$ such that
\[
    \irr\left(\mathcal{A}_{(1,d)}\right)
    \leq C_{\varepsilon}\cdot d^{2+\varepsilon}.
\]
\end{itemize}
\end{thm}

Our second main result concerns moduli spaces of projective hyperk\"{a}hler manifolds. The first examples of such manifolds, introduced in \cite{Bea83}, are Hilbert schemes of points on K3 surfaces and generalized Kummer varieties. A hyperk\"{a}hler manifold $X$ is said to be of K3$^{[n]}$-type or Kum$_n$-type if it is deformation equivalent to, respectively,  the former or the latter. In addition, there are two sporadic examples in dimensions $10$ and~$6$ constructed in \cite{OGr99, OGr03}. In this case,  $X$ is of OG10- and OG6-type, respectively. Up to now, every known projective hyperk\"{a}hler manifold is of one of these four types.

For a projective hyperk\"{a}hler manifold $X$, its second cohomology group $H^2(X,\mathbb{Z})$ carries a bilinear form $(\cdot,\cdot)$ which turns it into a lattice known as the \emph{Beauville--Bogomolov--Fujiki lattice}. If the manifold $X$ is of one of the known deformation types, then this lattice is isomorphic to one of the lattices $\Lambda_{\mathrm{K3}^{[n]}}$, $\Lambda_{\mathrm{Kum}_n}$, $\Lambda_{\mathrm{OG10}}$, $\Lambda_{\mathrm{OG6}}$. If $\Lambda$ is one of these lattices, we denote by $\mathcal{M}_{\Lambda,2d}^\gamma$ the moduli space of pairs $(X,H)$, where $X$ is a projective hyperk\"{a}hler manifold with $H^2(X,\bZ)\cong\Lambda$, and $H$ is a primitive polarization on $X$ of degree $(c_1(H),c_1(H)) = 2d$ and divisibility $\gamma$ (recall that the \emph{divisibility} of $x\in\Lambda$ is the positive generator of the ideal $(x,\Lambda)\subset\mathbb{Z}$). Note that the dimension of $X$, which equals $2n$ for some integer $n$, can be read off $\Lambda$. The existence of such moduli spaces follows from \cite{Vie95}, and their irreducible components are birational to orthogonal modular varieties due to the Torelli theorem; see \cite{Ver13} (see also \cite{Mar11}). This lays the foundation for our second main result about their degrees of irrationality.

\begin{thm}
There exists a constant $C>0$ such that, for every irreducible component $Y\subset\mathcal{M}_{\Lambda,2d}^\gamma$, it holds that
$$
    \irr(Y)\leq C\cdot (n\cdot d)^{19}. 
$$
If we consider only \emph{Kum}$_n$-type hyperk\"{a}hler manifolds, then this bound can be refined as
$$
    \irr(Y)\leq C\cdot(n\cdot d)^{11}.
$$
Furthermore, for any $\varepsilon>0$, there exists a constant $C_\varepsilon$ such that for every irreducible component $Y\subset\mathcal{M}_{\Lambda,2d}^\gamma$, it holds that
\begin{itemize}
\item $\irr(Y) \leq C_\varepsilon\cdot d^{14+\varepsilon}$ if we consider only \emph{OG10}-type hyperk\"{a}hler manifolds,
\item $\irr(Y) \leq C_\varepsilon\cdot d^{6+\varepsilon}$ if we consider only \emph{OG6}-type hyperk\"{a}hler manifolds.
\end{itemize}
\end{thm}

There are special series of $n$ and $d$ for which one can considerably improve the bound; see Theorems~\ref{thm:irratK3n},~\ref{thm:irratKumn},~\ref{thm:OG10_bound}, and~\ref{sec3:them:OG6}.


Our approach is the same as in \cite{ABL23} but with several additional challenges. First of all, the Torelli theorem states that every irreducible component $Y\subset \mathcal{M}_{\Lambda,2d}^{\gamma}$ admits an open embedding into an orthogonal modular variety associated with an even lattice $M$ of signature~$(2,m)$ for some $m>0$. Hence the degree of irrationality of $Y$ coincides with that of $\Omega(M)/\Gamma$, where $\Omega(M)$ is the period domain of $M$ and $\Gamma\subset\uO^+(M)$ is an arithmetic group.
Here $M$ and $\Gamma$ depend on the deformation type, the degree, the divisibility, as well as on the irreducible component $Y$. In each setting, this gives a collection of lattices $\Lambda_{n,d}$ of signature $(2,m)$ and arithmetic groups $\Gamma_{n,d}\subset \uO^+\left(\Lambda_{n,d}\right)$.

Assume that there exists an even lattice $\Lambda_{\#}$ of signature $(2,m')$ independent of $n,d$ and embeddings $\Lambda_{n,d}\hookrightarrow\Lambda_{\#}$ for all $n,d$. Assume further that each $\Gamma_{n,d}$ is \emph{extendable}, that is, each $g\in\Gamma_{n,d}$ can be extended as an isometry $g_\#\in \uO^+(\Lambda_\#)$ preserving $\Lambda_{n,d}$ whose restriction to $\Lambda_{n,d}$ recovers $g$. These assumptions induce morphisms of quasi-projective varieties
\begin{equation}\label{eq:mapfnd}
    \quad  \mathcal{P}_{n,d} \colonequals \Omega(\Lambda_{n,d})\big/\Gamma_{n,d}
    \xrightarrow{f_{n,d}}
    \Omega(\Lambda_{\#})\big/\uO^+(\Lambda_{\#}) =: \mathcal{P}^+_{\Lambda_{\#}}
\end{equation}
of finite degree onto their images $Z_{n,d} \subseteq \mathcal{P}^+_{\Lambda_{\#}}$.
In particular, there is the immediate inequality
\[
\irr\left(\mathcal{P}_{n,d}\right)\leq \deg(f_{n,d})\cdot\irr\left(Z_{n,d}\right).
\]

The cycles $[Z_{n,d}]$
are examples of {\textit{special cycles}} (also referred to as \textit{Kudla cycles}). They can be arranged in a generating series which turns out to be the Fourier expansion of a modular form. More precisely, if we fix an embedding
$
    \mathcal{P}^+_{\Lambda_{\#}}
    \hookrightarrow\bP^N
$
and let $\overline{Z}_{n,d}$ be the closure of $Z_{n,d}$ in $\mathbb{P}^N$, then Kudla's modularity conjecture, see \cite{Kud97, Kud04}, proved in \cite{Kud04,Zha09,BW15}, implies that the integers $\deg\left(\overline{Z}_{n,d}\right)$ are coefficients of the Fourier expansion of a Siegel modular form of weight $\frac{1}{2}\operatorname{rk}(\Lambda_{\#})$. By taking a projection onto a general linear subspace $\bP^{\dim(Z_{n,d})}\subset\bP^N$, one can conclude that $\irr(Z_{n,d})\leq \deg(\overline{Z}_{n,d})$, and this can be estimated via
 standard bounds on the growth of coefficients of Siegel modular forms.

There are two main challenges: The first one is to find an appropriate $\Lambda_{\#}$ together with embeddings $\Lambda_{n,d}\hookrightarrow\Lambda_{\#}$ such that the $\Gamma_{n,d}$ are extendable. The second one is to bound $\deg(f_{n,d})$. The general strategy was developed in \cite{ABL23}, and the bulk of the paper is devoted to overcoming these two challenges.

\subsection*{Outline} This paper is organized as follows. In Section~\ref{sect:orthShimuraIrr}, we set up notation and give bounds on degrees of irrationality for special cycles. Then we establish bounds on degrees of maps onto their images between orthogonal modular varieties induced by lattice embeddings. In Section~\ref{sect:known-proj-HK}, we apply the results in Section~\ref{sect:orthShimuraIrr} to study the irrationality of moduli spaces of projective hyperk\"{a}hler varieties of known deformation types. Finally, in Section~\ref{sec:Ab_surf}, we study the irrationality of the moduli space of $(1,d)$-polarized abelian surfaces and also revisit the case of K3 surfaces.

\subsection*{Acknowledgements} The authors would like to thank Emma Brakkee, Simon Brandhorst, Nathan Chen, Laure Flapan, Sam Grushevsky, Klaus Hulek, and Christopher Voll for helpful comments and stimulating conversations. Special thanks go to Giacomo Mezzedimi for his idea in the HK of K3$^{[n]}$-type of divisibility $2$ case and to Pietro Beri for kindly pointing out a mistake in a previous version of this paper. Finally, we would also like to thank the anonymous
referees for their many helpful comments and suggestions. 

\section{Period spaces, special cycles, and irrationality}
\label{sect:orthShimuraIrr}

The aim of this section is to introduce some upper bounds on the degrees of irrationality of period spaces  needed in later sections. The lemmas and notation established along the way will also be used later.

\subsection{Degrees of irrationality of special cycles}
\label{subsect:sp-cycle}

Let us start by reviewing the main results in \cite{ABL23} about degrees of irrationality of special cycles on orthogonal Shimura varieties. Let $M$ be an even lattice of signature~$(2,m)$ for some $m>0$ (the assumption of \cite{ABL23} that $m$ is even was made for simplicity and is not necessary) and identify the dual lattice $M^\vee\colonequals\Hom(M,\bZ)$ as the space of vectors $v\in M\otimes_{\mathbb{Z}}\bQ$ satisfying $(v, M)\subset\bZ$. The discriminant group of~$M$ will be denoted as $D(M) \colonequals M^\vee/ M$; recall that its order is equal to the absolute value of the discriminant $\operatorname{disc}(M)$. Let us further define $\uO(M)$ to be the group of isometries of $M$. Then the isometries which preserve the orientation of one (and thus all) positive $2$-plane in $M\otimes\bR$ form an index two subgroup $\uO^+(M)\subset\uO(M)$.

The period domain $\Omega(M)$ is defined as one of the two components of
\[
    \left\{
        [w]\in\bP( M\otimes\bC)
        \;\middle|\;
        (w,w)=0,\;(w,\overline{w})>0
    \right\}.
\]
We consider this together with the natural action of a finite-index subgroup $\Gamma\subset\uO^+(M)$. Then, the {\textit{period space}}
\[
    \cP_M(\Gamma)
    \colonequals\Omega(M)/\Gamma
\]
is an orthogonal Shimura variety that is quasi-projective; \textit{cf.} \cite{BB66}. More precisely, the restriction of $\cO_{\bP(M\otimes\bC)}(-1)$ to $\Omega(M)$ descends to a $\mathbb{Q}$-line bundle $\mathcal{L}$ on $\mathcal{P}_M(\Gamma)$ such that the space $H^0(\mathcal{L}^{\otimes k})$, where $k>1$ is such that $\mathcal{L}^{\otimes k}$ is a line bundle, can be identified with the space $\mathrm{Mod}_k(\Gamma,1)$ of modular forms for $\Gamma$ of weight~$k$ and trivial character. A projective model of $\mathcal{P}_M(\Gamma)$ is then given by the Baily--Borel compactification
$$
    \overline{\mathcal{P}_M(\Gamma)}^{BB}
    = \mathrm{Proj}\left(
        \bigoplus_{k\geq 0}H^{0}(\mathcal{L}^{\otimes k})
    \right)
    = \mathrm{Proj}\left(
        \bigoplus_{k\geq 0}\mathrm{Mod}_k(\Gamma,1)
    \right).
$$
It is normal, and the complement of $\mathcal{P}_M(\Gamma)$ has dimension at most one.

For every $t\in\bQ_{\geq 0}$ and $\gamma\in D(M)$, we consider the formal sum of hyperplane sections
$$
    \sum_{
        \frac{1}{2}(v,v)=t,\; v\equiv\gamma
    }v^\perp
    \;\subset\;
    \Omega(M), 
$$
where the sum runs through all $v\in M^\vee$ and
$
    v^\perp\colonequals\left\{
        [w]\in \Omega\left( M\right)
        \;\middle|\;
        (w,v)=0
    \right\}.
$
Under the map $\Omega\left(M\right)\rightarrow\cP_M\left(\Gamma\right)$, the formal sum descends to a $\mathbb{Q}$-Cartier divisor $Y_{t,\gamma}\subset\cP_M(\Gamma)$ called a \emph{Heegner divisor}. Let us fix a projective embedding $\overline{\cP_M(\Gamma)}^{BB}\subset\bP$, and let $\overline{Y}_{t,\gamma}\subset\overline{\cP_M(\Gamma)}^{BB}$ be the closure of $Y_{t,\gamma}$. By fundamental results of Kudla--Millson \cite{KM90} and Borcherds \cite{Bor99}, for any $\gamma$, the series
\[
\sum_{t\in \mathbb{Q}_{\geq 0}}{\rm{deg}}\left(\overline{Y}_{t,\gamma}\right)q^t
\]
is the Fourier expansion of a scalar-valued modular form of weight $\frac{m}{2}+1$; see also \cite{Bru02} for details. The growth of the Fourier coefficients of such a series are bounded by $C\cdot t^\frac{m}{2}$ for some positive constant $C>0$. Since the degree of irrationality of a variety is at most the degree of the variety under a projective embedding, the same argument as in the proof of \cite[Theorem~2.6]{ABL23} yields the following.

\begin{lemma}
\label{lem:irr_Heegner}
Assume $m\geq 3$. Then there exists a constant $C>0$ such that for all $t\in\bQ_{>0}$ and $\gamma\in D(M)$, it holds that
$$
    {\rm{irr}}\left(\overline{Y}_{t,\gamma}\right)
    \leq C\cdot t^{\frac{m}{2}}.
$$
\end{lemma}

As a higher-codimensional analogue, for each $r$-tuple
$
    \underline{v} = (
        v_1,\dots,v_r
    )\in\left(
        M^\vee
    \right)^{\oplus r},
$
one can consider the linear subspace $\langle\underline{v}\rangle\subset\bP(M\otimes\bC)$ spanned by $v_1,\dots,v_r$ and the moment matrix $Q(\underline{v}) = \frac{1}{2}((v_i,v_j))$. Fix a semi-positive symmetric matrix $T\in{\rm{Sym}}_r\left(\bQ\right)_{\geq 0}$ and an $r$-tuple of classes in the discriminant group $\underline{\gamma}\in D( M)^{\oplus r}$. Then the formal sum 
$$
    \sum_{
        \substack{Q(\underline{v})=T,\;
        \underline{v}\equiv \underline{\gamma}}
    }\langle
        \underline{v}
    \rangle^\perp
    \subset\Omega\left(M\right),
$$
which runs through all
$
    \underline{v}\in\left(
        M^\vee
    \right)^{\oplus r},
$
descends to a cycle $Z_{T,\underline{\gamma}}\subset\cP_{M}\left(\Gamma\right)$. Borcherds' result on Heegner divisors has a generalization known as Kudla's modularity conjecture, which is proved in a series of works \cite{Kud97,Kud04,Zha09,BW15}. Using this result, we obtained in \cite[Theorem~6.2]{ABL23} the following bound.

\begin{lemma}
\label{lem:irr_Kudla}
Assume  $1\leq r\leq m-2$. Then there exists a constant $C>0$ such that for all $T\in{\rm{Sym}}_r(\mathbb{Q})_{>0}$ and $\underline{\gamma}\in D( M)^{\oplus r}$, it holds that
$$
    {\rm{irr}}\left(Z_{T,\underline{\gamma}}\right)\leq C\cdot {\rm{det}}\left(T\right)^{1+\frac{m}{2}}.
$$
\end{lemma}

\subsection{Morphisms induced by lattice embeddings}
\label{subsect:Hdg-partner}

Let $\Lambda$ and $\Lambda_{\#}$ be even lattices of signatures $(2,m)$ and $(2,m')$ with $3\leq m\leq m'$. Suppose there is an embedding $\Lambda \hookrightarrow \Lambda_{\#}$, not necessarily primitive. We say that a finite-index subgroup $\Gamma\subset \uO^+(\Lambda)$ is \emph{extendable} with respect to this embedding if for every $g\in \Gamma$, there exists a $g_{\#} \in \uO^+(\Lambda_{\#})$ such that $g_{\#}(\Lambda)\subset\Lambda$ and $g_{\#}|_{\Lambda} = g$. In this setting, we have a natural map
\begin{equation}
\label{map:quasi-finite_hodge_gamma}
    \psi_\Gamma\colon
    \cP_{\Lambda}(\Gamma)
    = \Omega(\Lambda)/\Gamma
\longrightarrow
    {\cP}^+_{\Lambda_\#}
    \colonequals
    \Omega(\Lambda_\#)/{\uO}^+(\Lambda_\#)
\end{equation}
with finite fibers. This map is actually a morphism of algebraic varieties due to Borel's extension theorem \cite[Theorem~3.10]{Bor72}.

In the following, we are going to bound the degree of $\psi_\Gamma$ onto its image. To do so we consider the \emph{stable orthogonal group}
$$
    \widetilde{\uO}(\Lambda)\colonequals\left\{
        g\in\uO(\Lambda)
    \;\middle|\;
        g\text{ acts trivially on }D(\Lambda)
    \right\}
$$
and define
$
    \widetilde{\uO}^+(\Lambda)\colonequals
    \uO^+(\Lambda)\cap\widetilde{\uO}(\Lambda).
$
Note that this is a finite-index subgroup of $\uO^+(\Lambda)$.

\begin{lemma}
\label{lem:O+tilde}
There exists an injection
$
    \widetilde{\uO}^+(\Lambda)
    \hookrightarrow
    \widetilde{{\uO}}^+(\Lambda_\#)
$
given by extending $g\in\widetilde{\uO}^+(\Lambda)$ to $\Lambda_\#$ as the identity on $\Lambda^{\perp\Lambda_\#}$. In particular, the group $\widetilde{\uO}^+(\Lambda)$ is extendable with respect to the embedding $\Lambda\hookrightarrow\Lambda_{\#}$.
\end{lemma}

\begin{proof}
By the proof of \cite[Lemma~7.1]{GHS13}, one can embed ${\uO}^+(\Lambda)$ into $\widetilde{\uO}(\Lambda_\#)$ by extending $g\in\widetilde{\uO}^+(\Lambda)$ to~$\Lambda_\#$ as the identity on $\Lambda^{\perp\Lambda_\#}$. Note that a positive $2$-plane in $\Lambda\otimes\mathbb{R}$ corresponds to a positive $2$-plane in $\Lambda_\#\otimes\mathbb{R}$. Hence the extension of $g$ is orientation-preserving; that is, it lies in $\uO^+(\Lambda_{\#})$. This completes the proof.
\end{proof}

Lemma~\ref{lem:O+tilde} allows us to insert $\Gamma = \widetilde{\uO}^+(\Lambda)$ into $\psi_\Gamma$, which defines a morphism
$$
    \widetilde{\psi}^+\colon
    \widetilde{\cP}^+_{\Lambda}
    \colonequals
    \Omega(\Lambda)/\widetilde{\uO}^+(\Lambda)
    \longrightarrow
    {\cP}^+_{\Lambda_\#}
    = \Omega(\Lambda_\#)/{\uO}^+(\Lambda_\#).
$$
Our first step is to bound the degree of this map onto its image assuming $\Lambda\hookrightarrow \Lambda_{\#}$ is a primitive embedding. Recall that the points of $\Omega(\Lambda)$ correspond to \emph{Hodge structures of K3 type} on the lattice $\Lambda$. For each $[v]\in\Omega(\Lambda)$, the \emph{transcendental lattice} $T(v)\subset\Lambda$ is the minimal primitive sub-Hodge structure (in particular, $T(v)\subset \Lambda$ is primitive) of the Hodge structure induced by $[v]$ with $(T(v)\otimes\bC)^{2,0} = \bC v$.

\begin{lemma}
\label{lem:trans-is-all}
For a very general $[v]\in\Omega(\Lambda)$, we have $T(v) = \Lambda$.
\end{lemma}

\begin{proof}
For a very general $[v]\in\Omega(\Lambda)$, the hyperplane section $v^{\perp}\subset\Lambda\otimes\bC$ contains no element from $\Lambda$ because $\Lambda$ consists of countably many points. In this situation, the Hodge structure on $\Lambda$ determined by $[v]$ satisfies $\Lambda^{1,1} = \{0\}$, whence $T(v) = (\Lambda^{1,1})^\perp = \Lambda$.
\end{proof}

The following lemma shows that transcendental lattices are unchanged under primitive embeddings of lattices.

\begin{lemma}
\label{lem:trans-inv}
Consider a primitive embedding $\iota\colon\Lambda\hookrightarrow\Lambda_{\#}$. For every $[v]\in\Omega(\Lambda)$, we have
$$
    \iota(T(v)) = T(\iota(v)).
$$
That is, the transcendental lattice on $\Lambda$ determined by $[v]$ is mapped isomorphically onto the transcendental lattice on $\Lambda_{\#}$ determined by $[\iota(v)]$.
\end{lemma}

\begin{proof}
The Hodge substructure $\iota(T(v))\subset\Lambda_\#$ has $(2,0)$-part spanned by $\iota(v)$. This implies 
$
    \iota(T(v)) \supset T(\iota(v))
$
due to the minimality of $T(\iota(v))$. Now, $T(\iota(v))$ appears as a Hodge substructure of $\iota(\Lambda)$ with $(2,0)$-part spanned by $\iota(v)$, so the containment is an equality due to the minimality of $\iota(T(v))$.
\end{proof}

The next lemma gives a slight extension of \cite[Lemma~4.3]{OS18}.

\begin{lemma}
\label{lem:T(v_1)=T(v_2)}
Suppose that $[v_1],[v_2]\in\Omega(\Lambda)$ have the same image under $\widetilde{\psi}^+$. Then there exists a $g\in{\uO}^+(\Lambda_\#)$ such that there is an identification $g(T(v_1)) = T(v_2)$ of Hodge structures.
\end{lemma}

\begin{proof}
By Lemma~\ref{lem:trans-inv}, we can view $T(v_1)$ (resp.\ $T(v_2)$) as the transcendental lattice of the Hodge structure on~$\Lambda_\#$ defined by $v_1$ (resp.\ $v_2$). By hypothesis, there exists a $g\in{\uO}^+(\Lambda_\#)$ such that $g([v_1]) = [v_2]$, so it induces an isomorphism between the Hodge structures on $\Lambda_\#$ defined by $[v_1]$ and $[v_2]$. Thus $g(T(v_1)) = T(v_2)$ as they are minimal Hodge substructures associated to $g([v_1]) = [v_2]$.
\end{proof}

\begin{lemma}
\label{lem:same-image}
Let  $[v_1],[v_2]\in\Omega(\Lambda)$ be very general points such that there are identities ${T(v_1) = T(v_2) = \Lambda}$ as lattices, and assume that they are mapped to the same point under $\widetilde{\psi}^+$. Then there exists a $g\in \uO^+(\Lambda)$ such that $g([v_1])=[v_2]$.
\end{lemma}

\begin{proof}
By Lemma~\ref{lem:T(v_1)=T(v_2)}, there exists a $g\in\widetilde{\uO}^+(\Lambda_\#)$ such that $g(T(v_1)) = T(v_2)$ as Hodge structures (so that $g([v_1]) = [v_2]$). The hypothesis $T(v_1) = T(v_2) = \Lambda$ implies that the sublattice $\Lambda\subset\Lambda_\#$ is preserved by $g$, which gives an element $g|_\Lambda\in\uO(\Lambda)$. We need to check that $g|_{\Lambda}\in\uO^+(\Lambda)$.  Since it maps $[v_1]$ to $[v_2]$, it takes the positive $2$-plane in $\Lambda\otimes\bR$ spanned by $\mathrm{Re}(v_1), \mathrm{Im}(v_1)$ to the positive $2$-plane spanned by $\mathrm{Re}(v_2), \mathrm{Im}(v_2)$ and preserves their natural orientations. Hence $g|_{\Lambda}\in\uO^+(\Lambda)$. 
\end{proof}

\begin{lemma}
\label{lem:bound-Hdg-partners}
When the embedding $\Lambda \hookrightarrow \Lambda_{\#}$ is primitive, the degree of $\widetilde{\psi}^+$ onto its image is less than or equal to $|\uO(D(\Lambda))|$.
\end{lemma}

\begin{proof}
By Lemma~\ref{lem:same-image}, a very general fiber of $\widetilde{\psi}^+$ is contained in an orbit of $\uO^+(\Lambda)$ acting on $\Omega(\Lambda)/\widetilde{\uO}^+(\Lambda)$. Such an orbit has cardinality at most $|\uO^+(\Lambda)/\widetilde{\uO}^+(\Lambda)|$, which is in turn less than or equal to $|\uO(D(\Lambda))|$, so the statement follows.
\end{proof}

We want to extend  Lemma~\ref{lem:bound-Hdg-partners} to the situation when $\Lambda \hookrightarrow \Lambda_{\#}$ is not necessarily primitive. In this case, we let $\Lambda_s\subset\Lambda_{\#}$ be the saturation of $\Lambda$. This is again an even lattice of signature $(2,n)$, and the embedding $\Lambda_s \hookrightarrow \Lambda_{\#}$ is primitive.

\begin{lemma}
\label{lem:bound-Hdg-partners-unsat}
When the embedding $\Lambda \hookrightarrow \Lambda_{\#}$ is not necessarily primitive, the degree of $\widetilde{\psi}^+$ onto its image is less than or equal to $[\widetilde{\uO}^+(\Lambda_s):\widetilde{\uO}^+(\Lambda)]\cdot |\uO(D(\Lambda_s))|$.
\end{lemma}

\begin{proof}
According to Lemma~\ref{lem:O+tilde}, there are injections
$
    \widetilde{\uO}^+(\Lambda)
    \hookrightarrow\widetilde{\uO}^+(\Lambda_s)
    \hookrightarrow\widetilde{\uO}^+(\Lambda_{\#}).
$
This induces maps
$
    \widetilde{\cP}^+_{\Lambda}
    \to\widetilde{\cP}^+_{\Lambda_s}
    \to\cP^+_{\Lambda_{\#}}
$
whose composition is equal to $\widetilde{\psi}^+$. The first map has fibers of cardinality at most $[\widetilde{\uO}^+(\Lambda_s):\widetilde{\uO}^+(\Lambda)]$; the second map has degree onto its image bounded by $|\uO(D(\Lambda_s))|$ due to Lemma~\ref{lem:bound-Hdg-partners}. This proves the claim.
\end{proof}

We can finally give the bound for the map $\psi_\Gamma$ in \eqref{map:quasi-finite_hodge_gamma}.

\begin{lemma}
\label{lem:bound-Hdg-partners-gamma}
Let $\Lambda \hookrightarrow \Lambda_{\#}$ be an embedding and $\Gamma \subset \uO^+(\Lambda)$ be any extendable subgroup of finite index. Also let $\Lambda_s$ be the saturation of $\Lambda$ in $\Lambda_{\#}$. Then the degree of $\psi_\Gamma$ onto its image is less than or equal to
$
    [\widetilde{\uO}^+(\Lambda)
    : \Gamma\cap \widetilde{\uO}^+(\Lambda)]
    \cdot
    [\widetilde{\uO}^+(\Lambda_s)
    : \widetilde{\uO}^+(\Lambda)]
    \cdot
    |\uO(D(\Lambda_s))|.
$
\end{lemma}

\begin{proof}
Let $\widetilde{\Gamma}\colonequals\Gamma \cap \widetilde{\uO}^+(\Lambda)$. Then the map 
$
    \psi_{\widetilde{\Gamma}}
    \colon\cP_{\Lambda}(\widetilde{\Gamma})
    \rightarrow
    \cP^+_{\Lambda_{\#}}
$
factors as
$$
    \cP_{\Lambda}(\widetilde{\Gamma})
    \longrightarrow
    \cP_{\Lambda}(\Gamma)
    \overset{\psi_\Gamma}\longrightarrow 
    \cP^+_{\Lambda_{\#}}.
$$
Therefore, the degree of $\psi_\Gamma$ onto its image is bounded by that of $\psi_{\widetilde{\Gamma}}$. On the other hand, $\psi_{\widetilde{\Gamma}}$ also factors as
$$  \cP_{\Lambda}(\widetilde{\Gamma})
    \longrightarrow
    \widetilde{\cP}^+_{\Lambda}
    \overset{\;\widetilde{\psi}^+}\longrightarrow 
    \cP^+_{\Lambda_{\#}}.
$$
The first arrow is surjective with degree at most $[\widetilde{\uO}^+(\Lambda):\widetilde{\Gamma}]$. By Lemma~\ref{lem:bound-Hdg-partners-unsat}, the degree of $\widetilde{\psi}^+$ onto its image is bounded by $[\widetilde{\uO}^+(\Lambda_s):\widetilde{\uO}^+(\Lambda)]\cdot|\uO(D(\Lambda_s))|$. Hence the claim follows.
\end{proof}

In order to apply Lemma~\ref{lem:bound-Hdg-partners-gamma}, we need to bound each factor. For a lattice $M$, denote by $\ell(M)$ the minimal number of generators of the discriminant group $D(M)$. An element in $\uO(D(M))$ is determined by the images of the generators. Hence there is a bound
\[  
    |\uO(D(M))|
    \leq |\Aut(D(M))|
    \leq |\operatorname{disc}(M)|^{\ell(M)}.
\]

\begin{lemma}
\label{lem:boundorthogonalgroup}
Let $\Lambda\hookrightarrow\Lambda_{\#}$ be a not necessarily primitive embedding and $\Lambda_s\subset\Lambda_\#$ be the saturation of $\Lambda$. Then we have
$
    |\uO(D(\Lambda_s))|
    \leq |\operatorname{disc}(\Lambda)|^{\ell(\Lambda)}.
$
\end{lemma}

\begin{proof}
There is a chain of finite-index embeddings
$
    \Lambda
    \subset\Lambda_s
    \subset\Lambda_s^{\vee}
    \subset\Lambda^{\vee}.
$
In particular, we have equalities
$
    D(\Lambda_s)
    = \Lambda_s^{\vee}/\Lambda_s
    = (\Lambda_s^{\vee}/\Lambda)/(\Lambda_s/\Lambda),
$
where the last is a quotient of $\Lambda_{s}^{\vee}/\Lambda$ by a finite-index subgroup. Also notice  that $\Lambda_s^\vee/\Lambda\subset\Lambda^\vee/\Lambda = D(\Lambda)$. These facts imply
$\ell(\Lambda_s)\leq\ell(\Lambda)$ and $\mathrm{disc}(\Lambda_s)\leq\mathrm{disc}(\Lambda)$. Combining these inequalities gives the desired bound 
\begin{equation*}\pushQED{\qed}
    |\uO(D(\Lambda_s))|
    \leq |\operatorname{disc}(\Lambda_s)|^{\ell(\Lambda_s)}
    \leq |\operatorname{disc}(\Lambda)|^{\ell(\Lambda)}.
\qedhere \popQED
	\end{equation*}
\renewcommand{\qed}{}  
\end{proof}

Let us bound another term in Lemma~\ref{lem:bound-Hdg-partners-gamma}.

\begin{lemma}
\label{lemma:boundolambdas}
It holds that
$
    [\widetilde{\uO}^+(\Lambda_s):\widetilde{\uO}^+(\Lambda)]
    \leq [\Lambda_s:\Lambda]^{\operatorname{rk}(\Lambda)}
    \cdot |\uO(D(\Lambda))|.
$
\end{lemma}

\begin{proof}
Let $\uO(\Lambda_s, \Lambda)$ be the group of isometries $g\in \uO(\Lambda_s)$ such that $g(\Lambda)\subset\Lambda$. Then
$$
    [\widetilde{\uO}^+(\Lambda_s)
        : \widetilde{\uO}^+(\Lambda)]
    = [\widetilde{\uO}^+(\Lambda_s)
        : \widetilde{\uO}^+(\Lambda_s)
        \cap
        \uO(\Lambda_s,\Lambda)]
    \cdot
        [\widetilde{\uO}^+(\Lambda_s)
        \cap
        \uO(\Lambda_s,\Lambda)
        : \widetilde{\uO}^+(\Lambda)].
$$
By construction, $\uO(\Lambda_s,\Lambda)$ can be seen as a subgroup of $\uO(\Lambda)$, so the second factor on the right-hand side of the above equality is bounded by $[\uO(\Lambda):\widetilde{\uO}^+(\Lambda)] \leq |\uO(D(\Lambda))|$. For the first factor, observe that for every $g\in\widetilde{\uO}^+(\Lambda_s)$, the image $g(\Lambda)$ is a sublattice of $\Lambda_s$ of index $[\Lambda_s:\Lambda]$; furthermore, $g(\Lambda)=\Lambda$ if and only if $g\in\uO(\Lambda_s,\Lambda)$. This means that $\widetilde{\uO}^+(\Lambda_s)$ acts on the set $\Sigma$ of sublattices of $\Lambda_s$ of index equal to $[\Lambda_s:\Lambda]$ and with $\widetilde{\uO}^+(\Lambda_s)\cap \uO(\Lambda_s,\Lambda)$ as the stabilizer of $\Lambda$. Thus $[\widetilde{\uO}^+(\Lambda_s):\widetilde{\uO}^+(\Lambda_s)\cap \uO(\Lambda_s,\Lambda)] \leq |\Sigma| \leq [\Lambda_s:\Lambda]^{\operatorname{rk}(\Lambda)}$, where the second inequality comes from \cite[Remark~3.4]{Zo21}.
\end{proof}

\subsection{An even positive lattice of rank two}
\label{subsect:Qh}

There is a special lattice that will play an important role in our application, so we collect below some useful facts about it. Let $a,d$ be positive integers such that $a+d = 4t$ for some integer $t$. Then the lattice is given by
\begin{equation}
\label{eq:defQad}
    Q_{(a,d)}\colonequals
    \begin{pmatrix}
        2a & -a \\
        -a & 2t
    \end{pmatrix}.
\end{equation}
Denote by $z_1,z_2\in Q_{(a,d)}$ the basis vectors with respect to the above matrix. Consider the reflection isometry along $z_1$ given by
$$
    \sigma_{z_1}\colon Q_{(a,d)} \longrightarrow Q_{(a,d)}, \quad
     x\longmapsto x - 2\frac{(x,z_1)}{(z_1,z_1)}z_1.
$$
Our goal is to find an embedding of $Q_{(a,d)}$ into certain lattices that are defined independently of $a$ and $d$ such that $\sigma_{z_1}$ is extendable.

In the following, we will consider the sublattice
$$
    W = \left<w_1, w_2\right>\subset Q_{(a,d)}
    \quad\text{where }
    w_1\colonequals z_1
    \text{ and } 
    w_2\colonequals 2z_2 + z_1.
$$
Notice that $z_1 = w_1$ and $z_2 = \frac{1}{2}(w_2-w_1)$, so $W$ has index two in $Q_{(a,d)}$. Moreover, we have $w_1^2 = 2a$, $w_2^2 = 2d$, and $(w_1,w_2) = 0$, so $W$ is isomorphic to $\bZ(2a)\oplus\bZ(2d)$. Our strategy for finding a desired embedding for $Q_{(a,d)}$ starts with finding an embedding for $W$ with nice properties. We will need Lagrange's four square theorem, which states that any non-negative integer $a$ can be written as the sum of four squares, $a=a_1^2+a_2^2+a_3^2+a_4^2$. It turns out that $a$ can be written as the sum of four coprime squares if and only if $8$ does not divide $a$; see \cite[Theorem~1]{CH07} and also \cite{MO}.

\begin{lemma}
\label{lemma:qadnot4} 
Suppose that both $a$ and $d$ are not divisible by $4$. Then there exists a primitive embedding $Q_{(a,d)} \hookrightarrow E_8$ such that $\sigma_{z_1}$ is extendable. 
\end{lemma} 

\begin{proof}
Because $a$ and $d$ are not divisible by $4$, by \cite[Theorem~1]{CH07} each of them can be written as a sum of four coprime squares:
$
    a = a_1^2 + a_2^2 + a_3^2 + a_4^2
$
and
$
    d = b_1^2 + b_2^2 + b_3^2 + b_4^2.
$
The hypothesis also implies that the~$a_i$ (resp.\ the $b_i$) cannot be all even or all odd. We claim that, upon rearranging the indices, we can assume $b_i\pm a_i$ is odd for all $i$. Indeed, we can write $a + d = 4t$ as
$$
    \sum_{i=1}^4 a_i^2 + \sum_{j=1}^4 b_j^2
    \equiv 0 \mod 4.
$$
A direct check shows that the number of even $a_i$ is the same as the number of odd $b_i$. Hence, upon rearranging the indices, we can assume that $a_i$ is even if and only if $b_i$ is odd for each $i\in\{1,2,3,4\}$, which implies that $b_i\pm a_i$ is odd for each $i$.

Now consider  the vector space $\bigoplus_{i=1}^8\bQ e_i$ with the standard Euclidean product. Then the lattice $E_8$ can be realized as the sublattice in $\bigoplus_{i=1}^8\bQ e_i$ whose coordinates are either all integers or all half integers such that the sum of all coordinates is even. There is an embedding $W\hookrightarrow E_8$ given by
\begin{align*}
	w_1 &= a_1(e_1 - e_5)+ a_2(e_2 - e_6)+ a_3(e_3 - e_7)+ a_4(e_4 - e_8), \\
	w_2 &= b_1(e_1 + e_5)+ b_2(e_2 + e_6)+ b_3(e_3 + e_7)+ b_4(e_4 + e_8).
\end{align*}
Now we have
$$
	z_2 = \frac{1}{2} \left(w_2 - w_1\right) 
    = \sum_{i=1}^4\frac{b_i - a_i}{2}e_i
	+ \sum_{i=1}^4\frac{b_i + a_i}{2}e_{i+4}, 
$$
which belongs to $E_8$ since the $b_i\pm a_i$ are all odd. Hence the embedding extends to an embedding $Q_{(a,d)} \hookrightarrow E_8$. To check that $\sigma_{z_1}$ extends, consider the involution $\sigma \colon E_8 \to E_8$ defined by $\sigma(e_i)=e_{4+i}$ for $i=1,2,3,4$. Then $\sigma(w_1) = -w_1 = \sigma_{z_1}(w_1)$ and $\sigma(w_2) = w_2 = \sigma_{z_1}(w_2)$, so that $\sigma(z_1)=-z_1=\sigma_{z_1}(z_1)$ and $\sigma(z_2) = z_2+z_1 = \sigma_{z_1}(z_1)$. This shows that $\sigma$ is an extension of $\sigma_{z_1}$.

Let us prove that the embedding $Q_{(a,d)} \hookrightarrow E_8$ is primitive. Every $x\in E_8\cap (Q_{(a,d)}\otimes\bQ)$ can be written as
$$
	x = \sum_{i=1}^4 \left(
        n_i + \frac{\delta}{2}
    \right)e_i
    + \sum_{i=1}^4\left(
        n_{i+4} + \frac{\delta}{2}
    \right)e_{i+4}
    = \frac{A}{B}z_1 + \frac{C}{D} z_2, 
$$
where $n_i,n_{i+4} \in \mathbb{Z}$, $\delta\in \{0,1\}$, $A,B,C,D\in \mathbb{Z}$ with $A,B$ coprime and $C,D$ coprime. Expanding $z_1,z_2$ in $e_1,\dots, e_8$ and comparing the coefficients gives
$$
	n_i + \frac{\delta}{2}
    = \frac{A}{B}a_i + \frac{C}{2D}(b_i-a_i),
    \quad
    n_{i+4} + \frac{\delta}{2}
    = -\frac{A}{B}a_i + \frac{C}{2D}(b_i+a_i)
    \quad
    \text{for } i=1,\dots,4, 
$$
which implies
\begin{equation}
\label{eq:qadprimitive}
	n_i + n_{i+4} + \delta
    = \frac{C}{D} b_i,
    \quad
    n_i - n_{i+4}
    = \left(\frac{2A}{B} - \frac{C}{D}\right)a_i 
    \quad
    \text{for } i=1,\dots,4.
\end{equation}
The first set of equations in \eqref{eq:qadprimitive} implies that $D$ divides all the $b_i$. Since the $b_i$ are coprime, we get $D=1$. The second set of equations shows that $B$ divides $2a_i$ for $i=1,\dots,4$. Since the $a_i$ are coprime, we conclude that $B=1$ or $B=2$. If $B=1$, then $x\in Q_{(a,d)}$ and we are done. If $B=2$, then $A$ is odd. In this case, adding the two sets of equations in \eqref{eq:qadprimitive} together modulo $2$ gives
$$
    \delta
    \equiv C(b_i - a_i) + a_i
    \equiv C + a_i \mod 2, 
$$
where the second equality holds as the $b_i-a_i$ are all odd. Hence $a_i\equiv\delta - C \mod 2$, so the $a_i$ are all even or all odd, which gives a contradiction. Therefore, the embedding $Q_{(a,d)} \hookrightarrow E_8$ is primitive. This completes the proof.
\end{proof}

\begin{lemma}
\label{lemma:qad4}
Suppose that one, and hence both, of $a$ and $d$ are divisible by $4$. Then there exists an embedding $Q_{(a,d)}\hookrightarrow A_1^{\oplus 10}$ such that $\sigma_{z_1}$ is extendable. Moreover, if $Q_{(a,d),s}$ is the saturation, then $[Q_{(a,d),s}:Q_{(a,d)}] = 2$.
\end{lemma}

\begin{proof}
Let us write $\frac{a}{4}-1 = a_1^2+\dots+a_4^2$ and $\frac{d}{4}-1=b_1^2+\dots+b_4^2$ using Lagrange's four square theorem. Then there is an embedding $W\hookrightarrow A_1^{\oplus 10}$ defined by
\begin{align*}
	w_1 &= 2a_1 e_1 + 2a_2 e_2 + 2a_3 e_3 + 2a_4 e_4 + 2e_5, \\
	w_2 &= 2b_1 e_6 + 2b_2 e_7 + 2b_3 e_8 + 2b_4 e_9 + 2e_{10},
\end{align*}
where $\{e_1,\dots,e_{10}\}$ is the canonical basis for $A_1^{\oplus 10}$. Since $w_2-w_1$ is divisible by $2$, this extends to an embedding of $Q_{(a,d)}$ by taking $z_1 = w_1$ and $z_2 = \frac{1}{2}(w_2-w_1)$. The reflection $\sigma_{z_1} = \sigma_{w_1}$ extends to the involution $\sigma \colon A_1^{\oplus 10} \to A_1^{\oplus 10}$ defined by $\sigma(e_i)=-e_i$ and $\sigma(e_{5+i})=e_{5+i}$ for $i=1,\dots,5$. Finally, one can verify that the saturation is given by $Q_{(a,d),s} = \left<\frac{1}{2}z_1,z_2 \right>$. Indeed, if $p,q\in\mathbb{Q}$ are such that
$p\left(\frac{1}{2}z_1\right)+qz_2\in A_1^{\oplus 10}$, then the coefficient of $e_{10}$ is $q$, showing that $q\in\bZ$, and the coefficient of $e_5$ is $q-p$, showing that $p\in\bZ$. This proves the last assertion.
\end{proof}

Now consider the lattice $\mathbb{Z}(2a)\oplus \mathbb{Z}(2d)$ with canonical basis $\{z_1,z_2\}$, and let $\sigma_{z_1}$ denote the reflection along $z_1$. With a similar strategy, we can find an embedding for this lattice independent of $a$ and $d$ such that~$\sigma_{z_1}$ extends.

\begin{lemma}
\label{lemma:embeddingsplitlattice}
There exists a primitive embedding $\mathbb{Z}(2a)\oplus \mathbb{Z}(2d) \hookrightarrow A_1^{\oplus 10}$ such that $\sigma_{z_1}$ extends. When $8\nmid a$ and $8\nmid d$, there exists a primitive embedding $\mathbb{Z}(2a)\oplus \mathbb{Z}(2d) \hookrightarrow A_1^{\oplus 8}$ such that $\sigma_{z_1}$ extends. 
\end{lemma}

\begin{proof}
Write $a-1$ and $d-1$ as sums of four squares so that $a=a_1^2+a_2^2+a_3^2+a_4^2+1$ and $d = b_1^2+b_2^2+b_3^2+b_4^2+1$. Then the embedding of $\mathbb{Z}(2a)\oplus \mathbb{Z}(2d)$ into $A_1^{\oplus 10}$ is defined by
$$
    z_1 = a_1e_1+\dots+a_4e_4+e_5,
    \qquad z_2 = b_1e_5+\dots+b_4e_9+e_{10}.
$$
One can check directly that this is a primitive embedding. Moreover, $\sigma_{z_1}$ extends to the involution $\sigma\colon A_1^{\oplus 10} \to A_1^{\oplus 10}$ defined by $\sigma(e_i) = -e_i$ and $\sigma(e_{5+i}) = e_{5+i}$ for $i=1,\dots,5$.

If $8$ divides neither $a$ nor $d$, we can write them as the sum of four coprime squares $a = a_1^2+\dots+a_4^2$ and $d = b_1^2+\dots+b_4^2$. Then the embedding into $A_1^{\oplus 8}$ is defined by
$$
    z_1 = a_1e_1 + \dots + a_4e_4,
    \qquad z_2 = b_1e_5 + \dots + b_4e_8,
$$
which is again primitive. Furthermore, $\sigma_{z_1}$ extends to the involution $\sigma\colon A_1^{\oplus 8}\to A_1^{\oplus 8}$ defined by $\sigma(e_i)=-e_i$ and $\sigma(e_{4+i})=e_{4+i}$ for $i=1,\dots,4$. 
\end{proof}

\subsection{Degrees of irrationality of period spaces}
\label{subsect:irr-modular}

Putting the previous results together allows us to generalize the techniques in \cite[Section~6]{ABL23}. Suppose that $\Lambda$ is an even lattice of signature $(2,m)$ with an embedding into an even lattice $\Lambda_{\#}$ of signature $(2,m')$, where $1\leq m'-m\leq m'-2$. Consider an arithmetic subgroup $\Gamma\subset\uO^+(\Lambda)$ which is extendable to $\Lambda_{\#}$. We are going to bound the degree of irrationality of the period space $\mathcal{P}_{\Gamma}(\Lambda)$ in terms of this embedding. Let us first deal with the case when the embedding is primitive.

\begin{lemma}
\label{lemma:irrationality-general}
Suppose that the embedding $\Lambda \hookrightarrow \Lambda_{\#}$ is primitive, and consider the natural map
$
    \psi_{\Gamma}\colon\cP_{\Lambda}(\Gamma)
    \rightarrow
    {\cP}^+_{\Lambda_{\#}}.
$
Then there exists a constant $C$ depending only on $\Lambda_{\#}$ such that with $Y\colonequals\mathcal{P}_{\Lambda}(\Gamma)$, we have the inequalities
\begin{gather*}
    \irr(\psi_{\Gamma}(Y))\leq C\cdot
    |\operatorname{disc}(\Lambda)|^{1+\frac{m'}{2}},
\\
    \irr(Y)\leq C
    \cdot [\widetilde{\uO}^+(\Lambda)
        : \Gamma\cap\widetilde{\uO}^+(\Lambda)]
    \cdot |\uO(D(\Lambda))|
    \cdot |\operatorname{disc}(\Lambda)|^{1+\frac{m'}{2}}.
\end{gather*}
\end{lemma}

\begin{proof}
Choose a basis $\underline{w} = (w_1,\dots,w_{m'-m})$ for the orthogonal complement $\Lambda^{\perp\Lambda_{\#}}$ with moment matrix $T = \left(\frac{1}{2}(w_i, w_j)\right)$. Then $\Omega(\Lambda) = \langle \underline{w} \rangle^{\perp} \subset \Omega(\Lambda_{\#})$. Thus $\psi_{\Gamma}(Y)$ appears as a Zariski open subset of a component of the special cycle $Z_{T,0}\subset{\cP}^{+}_{\Lambda_{\#}}$. By Lemma~\ref{lem:irr_Kudla}, there exist a constant $C'$ depending only on $\Lambda_{\#}$ such that
$$
    \operatorname{irr}(\psi_{\Gamma}(Y))
    \leq C' \cdot |\det(T)|^{1+\frac{m'}{2}}.
$$
Note that (cf. \cite[Section~14.0.2]{Huy16})
$$
    \det(T)
    = \frac{1}{2^{m'-m}}\cdot|\operatorname{disc}(\Lambda^\perp)|
    = \frac{1}{2^{m'-m}}\cdot\frac{
            |\operatorname{disc}(\Lambda_{\#})|
        }{
            |\operatorname{disc}(\Lambda)|
        }
    \cdot\left[\Lambda_{\#}:\Lambda\oplus\Lambda^\perp\right]^2.
$$
From \cite[Section~7, Equation~(36)]{GHS13}, one can deduce that
$
    \left[\Lambda_{\#}:\Lambda\oplus\Lambda^\perp\right]
    \leq |\operatorname{disc}(\Lambda)|.
$
Hence
$$
    \left|\det(T)\right|
    \leq |{\operatorname{disc}(\Lambda_{\#})}|
    \cdot |\operatorname{disc}(\Lambda)|. 
$$
By setting $C \colonequals C'\cdot |\operatorname{disc}(\Lambda_{\#})|^{1+\frac{m'}{2}}$, we obtain the first inequality
$$
    \irr(\psi_{\Gamma}(Y))\leq C\cdot
    |\operatorname{disc}(\Lambda)|^{1+\frac{m'}{2}}.
$$
By Lemma~\ref{lem:bound-Hdg-partners-gamma}, we have
$$
    \irr(Y)\leq\irr(\psi_{\Gamma}(Y))\cdot
    [\widetilde{\uO}^+(\Lambda)
        : \Gamma\cap\widetilde{\uO}^+(\Lambda)]
    \cdot|\uO(D(\Lambda))|.
$$
Combining this with the previous inequality gives us the second inequality.
\end{proof}

Now let us deduce a bound when the embedding is not necessarily primitive.

\begin{lemma}
\label{lemma:irrationality-general-nonprimitie}
Let $\Lambda_s\subset\Lambda_\#$ be the saturation of $\Lambda$, and assume that $[\Lambda_s:\Lambda]\leq D$ for some constant $D$. Also let~$\ell(\Lambda)$ be the minimum number of generators for $D(\Lambda)$. Then there exists a constant $C$ depending only on $\Lambda_{\#}$ and $D$ such that with $Y=\mathcal{P}_{\Lambda}(\Gamma)$, it holds that
$$
    \irr(Y) \leq  C
    \cdot [\widetilde{\uO}^+(\Lambda)
        : \Gamma\cap \widetilde{\uO}^+(\Lambda)]
    \cdot |\operatorname{disc}(\Lambda)|^{1+\frac{m'}{2}+2\ell(\Lambda)}.
$$
\end{lemma}
\begin{proof}
By Lemma~\ref{lem:bound-Hdg-partners-gamma}, the degree of $\psi_{\Gamma}\colon \cP_{\Lambda}(\Gamma)\rightarrow\cP^+_{\Lambda_{\#}}$ onto its image is bounded by
$$
    [\widetilde{\uO}^+(\Lambda):\Gamma\cap \widetilde{\uO}^+(\Lambda)]
    \cdot [\widetilde{\uO}^+(\Lambda_s):\widetilde{\uO}^+(\Lambda)]
    \cdot |\uO(D(\Lambda_s))|.
$$
Lemma~\ref{lem:boundorthogonalgroup} shows that $|\uO(D(\Lambda_s))| \leq |\operatorname{disc}(\Lambda)|^{\ell(\Lambda)}$. Combining this with Lemma~\ref{lemma:boundolambdas} gives
\begin{align*}
    [\widetilde{\uO}^+(\Lambda_s)
        : \widetilde{\uO}^+(\Lambda)]
    \cdot |\uO(D(\Lambda_s))|
    & \leq [\Lambda_s:\Lambda]^{m+2}
    \cdot |\uO(D(\Lambda))| \cdot |\operatorname{disc}(\Lambda)|^{\ell(\Lambda)} \\
    & \leq D^{m+2} \cdot |\operatorname{disc}(\Lambda)|^{2\ell(\Lambda)}
    \leq D^{m'+2} \cdot |\operatorname{disc}(\Lambda)|^{2\ell(\Lambda)}.
\end{align*}
As a result, we obtain
$$
    \irr(Y) \leq \irr(\psi_{\Gamma}(Y))
    \cdot [\widetilde{\uO}^+(\Lambda)
        : \Gamma\cap \widetilde{\uO}^+(\Lambda)]
    \cdot D^{m'+2}
    \cdot |\operatorname{disc}(\Lambda)|^{2\ell(\Lambda)}.
$$
Recall from the proof of Lemma~\ref{lem:boundorthogonalgroup} that 
$
    |\operatorname{disc}(\Lambda_s)|
    \leq |\operatorname{disc}(\Lambda)|.
$
To bound $\irr(\psi_{\Gamma}(Y))$, first notice that $\cP_{\Lambda}(\Gamma) = \cP_{\Lambda_s}(\Gamma)$, so we can consider $\psi_{\Gamma}$ as a map from $\cP_{\Lambda_s}(\Gamma)$ to $\cP^+_{\Lambda_{\#}}$. Then Lemma~\ref{lemma:irrationality-general} shows that there exists a $C'>0$ depending only on $\Lambda_\#$ such that
$$
    \irr(\psi_{\Gamma}(Y))
    \leq C'\cdot |\operatorname{disc}(\Lambda_s)|^{1+\frac{m'}{2}} \leq C'\cdot |\operatorname{disc}(\Lambda)|^{1+\frac{m'}{2}}.
$$
Merging this into the above inequality with $C = C'\cdot D^{m'+2}$ gives the bound we want.
\end{proof}

\section{Moduli spaces of projective hyperk\"{a}hler manifolds}
\label{sect:known-proj-HK}

Let $X$ be a hyperk\"{a}hler manifold of types K3$^{[n]}$, Kum$_{n}$, OG6, or OG10. Then $H^2(X,\bZ)$ with the Beauville--Bogomolov--Fujiki form is an even lattice of signature $(3,b_2(X)-3)$. For each $h\in H^2(X,\bZ)$, one defines its \emph{divisibility} ${\rm div}_\Lambda(h)$ to be the positive generator of the ideal $(h, H^2(X,\bZ))\subset\bZ$. Let us fix a lattice~$\Lambda$ isometric to $H^2(X,\bZ)$, and let $\gamma$, $d$ be positive integers. Then the pairs $(X',H)$ of hyperk\"{a}hler manifolds equipped with a primitive ample divisor $H$ such that $H^2(X,\bZ)\cong\Lambda$ and $h\colonequals c_1(H)$ satisfies ${\rm div}_\Lambda(h) = \gamma$, $(h,h) = 2d$ form a moduli space $\mathcal{M}_{\Lambda,\,2d}^\gamma$ of dimension $b_2(X)-3$.

The monodromy group ${\rm Mon}^2(X)\subset\uO^+(H^2(X,\bZ))$ for all types of $X$ considered in this paper is normal and of finite index. Via a marking $\Lambda\cong H^2(X,\bZ)$, this defines a subgroup
$
    {\rm Mon}^2(\Lambda)\subset\uO^+(\Lambda)
$
independent of the choice of markings. Let $h\in\Lambda$ be a primitive ample class, and define $\Lambda_h\colonequals h^{\perp\Lambda}$. Then $\Lambda_h$ has signature $(2,b_2(X)-3)$ and discriminant
\begin{equation}
\label{sec4:eq:discLh}
    |\operatorname{disc}(\Lambda_h)|
    = \frac{2d\cdot |{\rm disc}(\Lambda)|}{\gamma^2}.
\end{equation}
The elements of ${\rm Mon}^2(\Lambda)$ fixing $h$ form a finite-index subgroup ${\rm Mon}^2(\Lambda, h)$, which can be identified as a subgroup of $\uO^+(\Lambda_h)$ by restriction. Let $Y\subset\mathcal{M}_{\Lambda, 2d}^{\gamma}$ be the irreducible component containing $X$. Then there exists an open embedding
$$
    Y\longhookrightarrow
    \Omega(\Lambda_h)\big/{\rm{Mon}}^2(\Lambda, h)
$$
as guaranteed by the Torelli theorem; see \cite{Ver13} and also \cite[Lemma~8.1]{Mar11}. This allows us to bound the degree of irrationality of $\mathcal{M}_{\Lambda, 2d}^{\gamma}$ using the results in the previous section. To obtain a universal bound for all the moduli spaces, what we need to do is to find embeddings of all possible $\Lambda_{h}$ into a common lattice $\Lambda_{\#}$ such that $\operatorname{Mon}^2(\Lambda,h)$ is extendable.

\subsection{Hyperk\"ahler manifolds of \texorpdfstring{K3$^{[n]}$}{K3n}-type}

Let $X$ be a hyperk\"ahler manifold deformation equivalent to the Hilbert scheme of length $n$ subschemes on a K3 surface. Here we assume that $n\geq 2$ since the case of K3 surfaces has already been treated in \cite{ABL23}. For such an $X$, the lattice $H^2(X,\mathbb{Z})$ is isomorphic to
$$
    \Lambda = \Lambda_{\mathrm{K3}^{[n]}}
    \colonequals
    E_8(-1)^{\oplus 2}
    \oplus U^{\oplus 3}
    \oplus\bZ\delta, 
    \quad\text{where }
    (\delta,\delta)=-2(n-1).
$$
This is an even lattice of signature $(3,20)$. According to \cite[Lemma~9.2]{Mar11}, we have
$$
    \mathrm{Mon}^2(\Lambda)
    = \widehat{\uO}^+(\Lambda)
    \colonequals\left\{
        g\in\uO^+(\Lambda)
        \;\middle|\;
        g|_{D(\Lambda)} = \pm\mathrm{id}
    \right\}.
$$
Take a primitive $h\in\Lambda$ with $(h,h)>0$. Then $h$ is in the same ${\rm{Mon}}^2(\Lambda)$-orbit as $\gamma(e+tf)-a\delta$ for suitable $t$ and $a$; see \cite[Lemma~3.4]{BBBF23}. Here $\{e,f\}$ is the standard basis of the first copy of $U$ in $\Lambda$. In particular, 
\[
    \Lambda_h \cong
    E_8(-1)^{\oplus 2}\oplus U^{\oplus 2}\oplus Q_h(-1), 
\]
where $Q_h(-1)\subset U\oplus\bZ\delta$ is a certain negative-definite rank two sublattice. We will need an explicit description of this lattice only in the cases of divisibility $\gamma=1,2$. 

\begin{lemma}
\label{lemma:qhk3n}
If\, $\gamma=1$, then $Q_h \cong \mathbb{Z}(2(n-1))\oplus \mathbb{Z}(2d)$. If $\gamma=2$, then $Q_h \cong Q_{(n-1,d)}$ as defined in \eqref{eq:defQad}.
\end{lemma}

\begin{proof}
This is an immediate consequence of \cite[Lemma~3.4 and Equation~(31)]{BBBF23}. Indeed, when $\gamma=1$, one can take $h=e+df$, and in this case, $Q_h(-1)$ is generated by $z_1=\delta$ and $z_2=e-df$. When $\gamma=2$, one can take $h=2(e+tf)-\delta$, where $t=\frac{d+(n-1)}{4}$ (this should be an integer; otherwise there is no corresponding hyperk\"ahler manifold; see \cite[Remark~3.3]{BBBF23}). In this case, $Q_h(-1)$ is generated by $z_1=(n-1)f-\delta$ and~$z_2=e-tf$.
\end{proof}

We keep the notation $\{z_1, z_2\}$ as generators of $Q_h(-1)$ in the cases described above. The lattice $\Lambda_h$ has signature $(2,20)$ and discriminant $\operatorname{disc}(\Lambda_h) = \frac{4d(n-1)}{\gamma^2}$. One can verify that by definition  $\gamma$ divides both $2d$ and $2(n-1)$. The monodromy group $\operatorname{Mon}^2(\Lambda,h)$ coincides with $\widehat{\uO}^+(\Lambda,h) = \uO(\Lambda,h)\cap \widehat{\uO}^+(\Lambda)$, which can then be described explicitly as follows.

\begin{lemma}
\label{lemma:monK3n}
Let $(n,d,\gamma)$ be as above with $\mathcal{M}_{K3^{[n]}, 2d}^\gamma$ non-empty.
\begin{enumerate}
\item If\, $n=2$ or $\gamma\geq 3$, then
$
    \widehat{\uO}^+(\Lambda,h)
    = \widetilde{\uO}^+(\Lambda_h).
$
\item If\, $n\geq 3$ and $\gamma=1,2$, then 
$
    \widehat{\uO}^+(\Lambda,h) = \langle
    \widetilde{\uO}^+(\Lambda_h), \sigma_{z_1}
    \rangle,
$
where $\sigma_{z_1}\in\uO^+(\Lambda_h)$ is the reflection along the primitive vector $z_1\in Q_h(-1)$.
\end{enumerate}
\end{lemma}
\begin{proof}
The first case is proved in \cite[Lemma~3.6 and Proposition~3.7]{BBBF23}. The second case is proved in \cite[Section~5]{BBBF23}.
\end{proof}

With this we can state the extendability result that we are going to need.

\begin{lemma}
\label{lemma:extendabilityK3n}
In each of the following cases, one can find an embedding $\Lambda_h\hookrightarrow\Lambda_{\#}$ for some even lattice $\Lambda_\#$ such that $\widehat{\uO}^+(\Lambda,h)$ extends:
\begin{enumerate}
\item\label{l:eK3n-1} If\, $n=2$ or $\gamma\geq 3$, then one can choose $\Lambda_{\#} = U^{\oplus 2}\oplus E_8(-1)^{\oplus 3}$ with the embedding primitive.
\item\label{l:eK3n-2} If\, $n\geq 3$ and $\gamma=1$, then one can choose $\Lambda_{\#} =U^{\oplus 2}\oplus E_8(-1)^{\oplus 2} \oplus A_1(-1)^{\oplus 10}$ with the embedding primitive.
\item\label{l:eK3n-3} If\, $n\geq 3$, $\gamma=1$, and $8\nmid n-1,8\nmid d$, then  $\Lambda_{\#} =U^{\oplus 2}\oplus E_8(-1)^{\oplus 2} \oplus A_1(-1)^{\oplus 8}$ and the embedding can be chosen primitive.
\item\label{l:eK3n-4} If\, $n\geq 3$, $\gamma=2$, and $4\nmid n-1,4\nmid d$, then  $\Lambda_{\#} =U^{\oplus 2}\oplus E_8(-1)^{\oplus 3}$ and the embedding can be chosen primitive.
\item\label{l:eK3n-5} If\, $n\geq 3$, $\gamma=2$, and $4\mid n-1,4\mid d$, then $\Lambda_{\#} =U^{\oplus 2}\oplus E_8(-1)^{\oplus 2} \oplus A_{1}(-1)^{\oplus 10}$. Let $\Lambda_{h,s}$ be the saturation. Then $[\Lambda_{h,s}:\Lambda_h]=2$.
\end{enumerate} 
\end{lemma}

\begin{proof}
Let us prove the statement case by case.

\eqref{l:eK3n-1}~
If $n=2$ or $\gamma\geq 3$, the group $\widehat{\uO}^+(\Lambda,h)$ coincides with the stable orthogonal group $\widetilde{\uO}^+(\Lambda_h)$ by Lemma~\ref{lemma:monK3n}. Hence, it is extendable with respect to any embedding. We can simply choose a primitive embedding $Q_{h}(-1) \hookrightarrow E_8(-1)$ (see \cite[Theorem~14.1.15]{Huy16}) to obtain a primitive embedding of $\Lambda_h$ into $U^{\oplus 2} \oplus E_8(-1)^{\oplus 3}$.

In all the other cases, Lemma~\ref{lemma:monK3n} shows that $\widehat{\uO}^+(\Lambda,h)=\langle \widetilde{\uO}^+(\Lambda_h),\sigma_{z_1} \rangle$. Suppose we can find an embedding $Q_h(-1)\hookrightarrow M$ such that $\sigma_{z_1}$ on $Q_h(-1)$ extends to $\sigma$ on $M$. Then we have an embedding of $\Lambda_h = U^{\oplus 2} \oplus E_8(-1)^{\oplus 2}\oplus Q_h(-1)$ into $U^{\oplus 2}\oplus E_8(-1)^{\oplus 2}\oplus M$ such that $\sigma_{z_1}$ extends to $\mathrm{Id}\oplus\sigma$, where $\mathrm{Id}$ is the identity map on $U^{\oplus 2}\oplus E_8(-1)^{\oplus 2}$. Since $\textrm{Id} \oplus \sigma$ acts as the identity on $U^{\oplus 2}$, it is orientation-preserving. Moreover, if we let $Q_h(-1)_s$ be the saturation of $Q_h(-1)$ inside $M$, then $[\Lambda_{h,s}:\Lambda_h] = [Q_{h}(-1)_{s}:Q_h(-1)]$. Thus, the problem is reduced to finding appropriate embeddings of $Q_h(-1)$.

\eqref{l:eK3n-2}~
If $n\geq 3$ and $\gamma=1$, Lemma~\ref{lemma:qhk3n} shows that $Q_h(-1) \cong \mathbb{Z}(-2(n-1))\oplus \mathbb{Z}(-2d)$. Then Lemma~\ref{lemma:embeddingsplitlattice} gives a primitive embedding $Q_h(-1)\hookrightarrow A_1^{\oplus 10}(-1)$ such that $\sigma_{z_1}$ extends.

\eqref{l:eK3n-3}~
If $n\geq 3$, $\gamma=1$, and $8\nmid n-1$, $8\nmid d$, Lemma~\ref{lemma:qhk3n} shows that $Q_h(-1) \cong \mathbb{Z}(-2(n-1))\oplus \mathbb{Z}(-2d)$. Then Lemma~\ref{lemma:embeddingsplitlattice} gives a primitive embedding $Q_h(-1)\hookrightarrow A_1^{\oplus 8}(-1)$ such that $\sigma_{z_1}$ extends.

\eqref{l:eK3n-4}~
If $n\geq 3$, $\gamma=2$, and $4\nmid n-1$, $4\nmid d$, Lemma~\ref{lemma:qhk3n} shows that $Q_h(-1) \cong Q_{(n-1,d)}(-1)$. Then Lemma~\ref{lemma:qadnot4} gives a primitive embedding $Q_h(-1)\hookrightarrow E_8(-1)$ such that $\sigma_{z_1}$ extends.

\eqref{l:eK3n-5}~
If $n\geq 3$, $\gamma=2$, and $4\mid n-1,4\mid d$, Lemma~\ref{lemma:qhk3n} shows that $Q_h(-1) \cong Q_{(n-1,d)}(-1)$. Then Lemma~\ref{lemma:qad4} gives an embedding $Q_h(-1)\hookrightarrow A_{1}^{\oplus 10}(-1)$ such that $\sigma_{z_1}$ extends. In this case, we have $[Q_h(-1)_s:Q_h(-1)] = 2$.

This completes the proof.
\end{proof}

We can now deduce our result for hyperk\"ahler manifolds of K3$^{[n]}$-type.

\begin{thm}
\label{thm:irratK3n}
There exists a constant $C>0$ such that for any $n,d,\gamma$ and for any irreducible component  $Y\subset\cM^\gamma_{K3^{[n]},2d}$, it holds that
$$
    \irr(Y)\leq C \cdot (n\cdot d)^{19}.
$$
Furthermore, for any $\varepsilon > 0$, there exists a constant $C_{\varepsilon}>0$ such that the above bound can be refined in each case as follows:
\begin{enumerate}
\item\label{t:iK3n-1} If\, $n=1$, then $\irr(Y) \leq C_{\varepsilon} \cdot (n\cdot d)^{14+\varepsilon}$.
\item\label{t:iK3n-2} If\, $n=2$ or $\gamma\geq 3$, then we have $\irr(Y)\leq C\cdot (n\cdot d)^{16}$.  If furthermore $\gamma$, $\frac{2d}{\gamma}$, $\frac{2(n-1)}{\gamma}$ are coprime, then $\irr(Y) \leq C_{\varepsilon}\cdot (n\cdot d)^{14+\varepsilon}$.
\item\label{t:iK3n-3} If\, $n\geq 3$ and $\gamma=1$, then we have $\irr(Y)\leq C_{\varepsilon}\cdot (n\cdot d)^{15+\varepsilon}$. If furthermore $8\nmid n-1$ and $8\nmid d$, then $\irr(Y) \leq C_{\varepsilon}\cdot (n\cdot d)^{14+\varepsilon}$.
\item\label{t:iK3n-4} If\, $n\geq 3$, $\gamma=2$, and $4\nmid n-1$, $4\nmid d$, then $\irr(Y) \leq C\cdot (n\cdot d)^{16}$.
\end{enumerate}
\end{thm}

\begin{proof}
  Recall that the component $Y$ is birational to $\Omega(\Lambda_h)/\operatorname{Mon}^2(\Lambda,h)$ for some $h\in\Lambda$. Let us start the proof by analyzing the situation in each case.

\eqref{t:iK3n-1}~
If $n=1$, note that $\cM^{\gamma}_{K3^{[1]}, 2d}$ is non-empty if and only if $\gamma = 1$, and in this case, it is irreducible. Then it was proven in \cite{ABL23} that for any $\varepsilon>0$, there exists a $C_{1,\varepsilon}>0$ such that for all $d>0$, it holds that
\begin{equation}
\label{eq:irratK3ncase1}
     \irr(\cM^{\gamma}_{K3^{[1]},2d})
     \leq C_{1,\varepsilon}\cdot (n\cdot d)^{14+\varepsilon}.
\end{equation}

\eqref{t:iK3n-2}~
If $n=2$ or $\gamma\geq 3$, Lemma~\ref{lemma:extendabilityK3n} shows that we can find a primitive embedding of $\Lambda_h$ into the lattice $\Lambda_{\#} = U^{\oplus 2} \oplus E_8(-1)^{\oplus 3}$ such that $\operatorname{Mon}^2(\Lambda,h)$ is extendable. Note that the lattice $\Lambda_{\#}$ is independent of $n,d,\gamma$, and $Y$. Then Lemma~\ref{lemma:irrationality-general} shows that there is a constant $C_2$ depending only on $\Lambda_{\#}$ such that
\begin{equation}
\label{eq:irratK3ncase2general} 
    \irr(Y)\leq C_2 \cdot |\uO(D(\Lambda_h))| \cdot |\operatorname{disc}(\Lambda_h)|^{14},
\end{equation}
where we use $[\widetilde{\uO}^+(\Lambda_h):\operatorname{Mon}^2(\Lambda,h) \cap \widetilde{\uO}^+(\Lambda_h)]=1$ from Lemma~\ref{lemma:monK3n}. 
On the other hand, $D(\Lambda_h) = D(Q_h(-1))$ is generated by at most two elements; hence Lemma~\ref{lem:boundorthogonalgroup} shows that $|\uO(D(\Lambda_h))| \leq |\operatorname{disc}(\Lambda_h)|^2$, thus
\begin{equation}
\label{eq:irratK3ncase2}
    \irr(Y) \leq C_2 
    \cdot |\operatorname{disc}(\Lambda_h)|^{16}
    \leq C_2 \cdot (n\cdot d)^{16},
\end{equation}
where the second inequality follows from $|\operatorname{disc}(\Lambda_h)| = \frac{4d(n-1)}{\gamma^2}$. If $\gamma,\frac{2d}{\gamma},\frac{2(n-1)}{\gamma}$ are coprime, then since ${\rm{O}}\left(\Lambda_h\right)\rightarrow{\rm{O}}\left(D\left(\Lambda_h\right)\right)$ is surjective (see \cite[Theorem 2.4]{Huy16}), by \cite[Proposition~3.12(ii)]{GHS10}, we have 
\[
|\uO(D(\Lambda_h))|\leq 2^{\rho\left(2(n-1)/\gamma\right)+1}\leq 2\cdot 2^{\rho\left(\frac{2(n-1)}{\gamma}\cdot\frac{2d}{\gamma}\right)},
\]
where $\rho(k)$ is the number of prime factors of $k$. It holds that $2^{\rho(k)}\leq \nu(k)={\rm{O}}(k^\varepsilon)$ for all $\varepsilon>0$, where $\nu(k)$ is the number of positive divisors of $k$; see \cite[Equation (6.4)]{ABL23}. Since $\operatorname{disc}(\Lambda_h)=\frac{4d(n-1)}{\gamma^2}$, this shows that for every $\varepsilon>0$, there exists a constant $C'_{2,\varepsilon}$ such that $|\uO(D(\Lambda_h))|\leq C'_{2,\varepsilon}\cdot |\operatorname{disc}(\Lambda_h)|^\varepsilon$. Plugging this into \eqref{eq:irratK3ncase2general} and setting $C_{2,\varepsilon} = C_2\cdot C'_{2,\varepsilon}$, we get the bound 
\begin{equation}
\label{eq:irratK3ncase2e}
    \irr(Y)\leq C_{2,\varepsilon}
    \cdot |\operatorname{disc}(\Lambda_h)|^{14+\varepsilon}
    \leq C_{2,\varepsilon} \cdot (n\cdot d)^{14+\varepsilon},
\end{equation}
where the second inequality comes from plugging in $|\operatorname{disc}(\Lambda_h)| = \frac{4d(n-1)}{\gamma^2}$.

\eqref{t:iK3n-3}~
If $n\geq 3$ and $\gamma= 1$, then Lemma~\ref{lemma:extendabilityK3n} shows that we have a primitive embedding of $\Lambda_h$ into $U^{\oplus 2} \oplus E_8(-1)^{\oplus 2} \oplus A_1(-1)^{\oplus 10}$ such that $\operatorname{Mon}^2(\Lambda,h)$ is extendable. Further, since $\gamma, \frac{2d}{\gamma},$ and $\frac{2(n-1)}{\gamma}$ are coprime in this case, by Lemma~\ref{lemma:irrationality-general} and \cite[Proposition~3.12(ii)]{GHS10}, we get a constant $C_{3,\varepsilon}$ such that
\begin{equation}
\label{eq:irratK3ncase3}
    \irr(Y) \leq C_{3,\varepsilon} \cdot (n\cdot d)^{15+\varepsilon}.
\end{equation}
If furthermore $8\nmid d$ and $8\nmid n-1$, then Lemma~\ref{lemma:extendabilityK3n} shows that we can find a primitive embedding of~$\Lambda_h$ into the lattice $U^{\oplus 2} \oplus E_8(-1)^{\oplus 2} \oplus A_1(-1)^{\oplus 8}$ such that $\operatorname{Mon}^2(\Lambda,h)$ is extendable. Moreover, \cite[Proposition~3.12(ii)]{GHS10} shows that for any $\varepsilon>0$, there exists a constant $C'_{3,\varepsilon}$ such that we have the inequality $|\uO(D(\Lambda_h))|\leq C'_{3,\varepsilon}\cdot |\operatorname{disc}(\Lambda_h)|^\varepsilon$. Hence, reasoning as in the proof of inequality~\eqref{eq:irratK3ncase2e}, we see that for any $\varepsilon>0$, there exists a $C_{3,\varepsilon}'>0$ such that
\begin{equation}
\label{eq:irratK3ncase3e}
    \irr(Y)
    \leq C_{3,\varepsilon}'
    \cdot (n\cdot d)^{14+\varepsilon}.
\end{equation}

\eqref{t:iK3n-4}~
Suppose $n\geq 3$ and $\gamma= 2$. If $4\nmid n-1$ and $4\nmid d$, Lemma~\ref{lemma:extendabilityK3n} shows that we can find a primitive embedding of $\Lambda_h$ into the lattice $U^{\oplus 2} \oplus E_8(-1)^{\oplus 3}$ such that $\operatorname{Mon}^2(\Lambda,h)$ is extendable. Then reasoning as in the proof of Equation~\eqref{eq:irratK3ncase2}, we see that there exists a constant $C_4$ such that
\begin{equation}
\label{eq:irratK3ncase4}
    \irr(Y) \leq C_4 \cdot (n\cdot d)^{16}.
\end{equation}
If $4\mid d$ and $4\mid n-1$, then Lemma~\ref{lemma:extendabilityK3n} shows that we can find an embedding of $\Lambda_h$ into the lattice $U^{\oplus 2} \oplus E_8(-1)^{\oplus 2} \oplus A_1(-1)^{\oplus 10}$ such that $\operatorname{Mon}^2(\Lambda,h)$ is extendable. The saturation $\Lambda_{h,s}$ of $\Lambda_h$ satisfies $[\Lambda_{h,s}:\Lambda_h]=2$; the group $D(\Lambda_h)=D(Q_h(-1))$ is generated by at most two elements; Lemma~\ref{lemma:monK3n} shows that $[\widetilde{\uO}^+(\Lambda_h):\operatorname{Mon}^2(\Lambda,h) \cap \widetilde{\uO}^+(\Lambda_h)]=1$. Hence, Lemma~\ref{lemma:irrationality-general-nonprimitie} asserts that there exists a constant $C_5>0$ such that
\begin{equation}
\label{eq:irratK3ncase5}
    \irr(Y) \leq C_5
    \cdot |\operatorname{disc}(\Lambda_h)|^{19}
    \leq C_5 \cdot (n\cdot d)^{19},
\end{equation}
where the second inequality comes from
$
    |\operatorname{disc}(\Lambda_h)|
    = \frac{4d(n-1)}{\gamma^2}.
$

Set $C=\max\{C_{1,5},C_2,C_{3,4},C_4,C_5\}$ and $C_{\varepsilon} = \max\{ C_{1,\varepsilon},C_{2,\varepsilon},C_{3,\varepsilon}, C_{3,\varepsilon}'\}$. The main inequality in the statement then follows from \eqref{eq:irratK3ncase1}, \eqref{eq:irratK3ncase2}, \eqref{eq:irratK3ncase3}, \eqref{eq:irratK3ncase4}, and \eqref{eq:irratK3ncase5}. The inequality in case~\eqref{t:iK3n-1} follows from \eqref{eq:irratK3ncase1}, the inequalities in case~\eqref{t:iK3n-2} follow from \eqref{eq:irratK3ncase2} and \eqref{eq:irratK3ncase2e}, the inequalities in case~\eqref{t:iK3n-3} follow from \eqref{eq:irratK3ncase3} and \eqref{eq:irratK3ncase3e}, and the inequality in case~\eqref{t:iK3n-4} follows from \eqref{eq:irratK3ncase4}.
\end{proof}

\subsection{Hyperk\"{a}hler manifolds of \texorpdfstring{$\boldsymbol{\operatorname{Kum}_n}$}{Kum-n}-type}

Let $X$ be a projective hyperk\"ahler manifold of dimension $2n$ deformation equivalent to the generalized Kummer variety of an abelian surface. In this case, the even lattice $H^2(X,\mathbb{Z})$ has signature $(3,4)$ and is isomorphic to
$$
    \Lambda = \Lambda_{\operatorname{Kum}_n}
    \colonequals U^{\oplus 3}\oplus \mathbb{Z}\eta,
    \quad\text{where }
    (\eta,\eta) = -2(n+1).
$$
According to \cite{Mon16} and \cite[Theorem 1.4]{Mar23}, the monodromy group $\mathrm{Mon}^2(\Lambda)$ is a subgroup of index at most two of $\widehat{\uO}^+(\Lambda)$ which contains
$
    \widetilde{\mathrm{SO}}^+\left(\Lambda\right)
    = \widetilde{\uO}^+\left(\Lambda\right)\cap{\mathrm{SO}}\left(\Lambda\right).
$
In particular, we have the following.

\begin{lemma}
\label{lemma:kumnhorbit}
Let $h\in\Lambda$ be a primitive vector of divisibility~$\gamma$ and positive square $(h,h)=2d$. Up to the action of $\operatorname{Mon}^2(\Lambda)$, we can assume $h=\gamma\left(e+tf\right)-a\eta$, where $\{e,f\}$ is the standard basis for the first copy of\, $U$ and $t,a$ are integers such that 
$$
    d = \gamma^2t-(n+1)a^2,
    \quad \gcd(a,\gamma)=1,
    \quad 0\leq a<\gamma.
$$
\end{lemma}

\begin{proof}
Let us prove that, up to the action of $\widetilde{\mathrm{SO}}^+\left(\Lambda\right)\subset \mathrm{Mon}^2(\Lambda)$, the vector $h$ can be expressed in the desired form. By Eichler's criterion (\textit{cf.} \cite[Proposition~2.15]{Son22}), two primitive vectors $h_1,h_2\in \Lambda$ are in the same $\widetilde{\mathrm{SO}}^+\left(\Lambda\right)$-orbit if and only if they have the same square $(h_1,h_1)=(h_2,h_2)=2d$, the same divisibility $\gamma$, and the same classes $\left[ \frac{h_1}{\gamma} \right] = \left[ \frac{h_2}{\gamma} \right]$ in $D(\Lambda)=D(\mathbb{Z}\eta) \cong \mathbb{Z}/2(n+1)\mathbb{Z}$. We can then proceed as in the proofs of \cite[Proposition 3.1 and Lemma 3.4]{BBBF23}, where we only need to replace $U^{\oplus 3}\oplus E_8(-1)^{\oplus 2}$ and $n-1$, respectively, by $U^{\oplus 3}$ and $n+1$.
\end{proof}

Let $h\in\Lambda$ be a primitive vector as in Lemma~\ref{lemma:kumnhorbit}. Then
$$
    \Lambda_h \cong U^{\oplus 2}\oplus Q_h(-1), 
$$
where $Q_h(-1)\subset U\oplus\bZ\eta$ is a certain negative-definite rank two sublattice which can be described explicitly as follows. 

\begin{lemma}
\label{lemma:qhkumn}
The rank two lattice $Q_h(-1)$ is generated by 
\begin{equation}
\label{lemma:eq:z_1z_2}
z_1:=\frac{2a(n+1)}{\gamma}f-\eta\quad\text{and}\quad z_2:=e-tf.
\end{equation}
In particular, if $\gamma=1$, then $Q_h \cong \mathbb{Z}(2(n+1))\oplus \mathbb{Z}(2d)$. If\, $\gamma=2$, then $Q_h \cong Q_{(n+1,d)}$ as defined in \eqref{eq:defQad}. If\, $\gamma\geq 3$, then 
\[
    Q_h = \begin{pmatrix}
        2(n+1) & -\frac{2a(n+1)}{\gamma} \\
        -\frac{2a(n+1)}{\gamma} & 2t
    \end{pmatrix}
\]
with $a,t,\gamma$ as in Lemma~\ref{lemma:kumnhorbit}.
\end{lemma}

\begin{proof}
This can be proved using Lemma~\ref{lemma:kumnhorbit} in the same way as one proves \cite[Lemma~3.4 and Equation~(31)]{BBBF23}, where one only needs to replace $U^{\oplus 3}\oplus E_8(-1)^{\oplus 2}$ and $n-1$, respectively, by $U^{\oplus 3}$ and~$n+1$.
\end{proof}

The lattice $\Lambda_h$ has signature $(2,4)$ and discriminant $\operatorname{disc}(\Lambda_h) = \frac{4d(n+1)}{\gamma^2}$. Note that $\gamma$ divides both $2d$ and $2(n+1)$. The monodromy group $\operatorname{Mon}^2(\Lambda,h)$ is a subgroup of index at most two of $\widehat{\uO}^+(\Lambda,h) = \uO(\Lambda,h)\cap \widehat{\uO}^+(\Lambda)$ which can be described explicitly as follows.

\begin{lemma}
\label{lemma:monkumn}
Let $(n,d,\gamma)$ be as above with $\mathcal{M}_{\mathrm{Kum}^{[n]},2d}^\gamma$ non-empty.
\begin{enumerate}
    \item If\, $\gamma\geq 3$, then
  	$
  	    \widehat{\uO}^+(\Lambda,h)
  	    = \widetilde{\uO}^+(\Lambda_h).
    $
  	\item If\, $\gamma=1,2$,  then 
  	$
  	    \widehat{\uO}^+(\Lambda,h) = \langle
  	        \widetilde{\uO}^+(\Lambda_h), \sigma_{z_1}
        \rangle,
    $
    where $\sigma_{z_1}\in\uO^+(\Lambda_h)$ is the reflection along $z_1\in Q_h(-1)$ of the basis \eqref{lemma:eq:z_1z_2}.
\end{enumerate}
\end{lemma}

\begin{proof}
This can be proved in the same way as \cite[Lemma~3.6 and Proposition~3.7]{BBBF23} with $\ell$ replaced by~$\eta$, $(n-1)$ replaced by $(n+1)$, and $U^{\oplus 2}\oplus E_8(-1)^{\oplus 2}\oplus Q_h(-1)$ replaced by $U^{\oplus 2}\oplus Q_h(-1)$.  
\end{proof}

As a consequence, we obtain an analogue of Lemma~\ref{lemma:extendabilityK3n}.

\begin{lemma}
\label{lemma:extendabilityKumn}
There exists a lattice $\Lambda_{\#}$, together with an embedding $\Lambda_h\hookrightarrow\Lambda_\#$, as follows such that $\widehat{\uO}^+(\Lambda,h)$ extends: 
\begin{enumerate}		
\item If\, $\gamma\geq 3$, then $\Lambda_{\#} = U^{\oplus 2}\oplus E_8(-1)$; the embedding is primitive.
\item If\, $\gamma=1$,  then $\Lambda_{\#} =U^{\oplus 2}\oplus A_1(-1)^{\oplus 10}$; the embedding is primitive.
\item If\, $\gamma=1$ and $8\nmid n-1,8\nmid d$, then  $\Lambda_{\#} =U^{\oplus 2} \oplus A_1(-1)^{\oplus 8}$; the embedding is primitive.
\item If\, $\gamma=2$ and $4\nmid n-1,4\nmid d$, then  $\Lambda_{\#} =U^{\oplus 2}\oplus E_8(-1)$; the embedding is primitive.
\item If\, $\gamma=2$ and $4\mid n-1,4\mid d$, then  $\Lambda_{\#} =U^{\oplus 2}\oplus  A_{1}(-1)^{\oplus 10}$. If $\Lambda_{h,s}$ is the saturation, then $[\Lambda_{h,s}:\Lambda_h]=2$.
	\end{enumerate} 
\end{lemma}

\begin{proof}
Since ${\operatorname{Mon}}^2\left(\Lambda,h\right)\subset \widehat{\uO}^+\left(\Lambda,h\right)$, it is enough to show the extendability of the latter. The proof is similar to the proof of Lemma~\ref{lemma:extendabilityK3n}, so we leave it to the reader.
\end{proof}

We are ready to prove our main result for hyperk\"ahler manifolds of Kum$_n$-type.

\begin{thm}
\label{thm:irratKumn}
There exists a constant $C>0$ such that for any $n,d,\gamma$ and for any irreducible component $Y\subset \mathcal{M}^\gamma_{\mathrm{Kum}_n,2d}$, it holds that
$$
    \irr(Y) \leq C \cdot (n\cdot d)^{11}.
$$
Furthermore, for every $\varepsilon>0$, there exists a constant $C_{\varepsilon}>0$ such that the above bound can be refined in each case as follows:
\begin{enumerate}
\item If $\gamma\geq 3$, then $\irr(Y)\leq C\cdot (n\cdot d)^{8}$. If furthermore $\gamma,\frac{2d}{\gamma},\frac{2(n+1)}{\gamma}$ are coprime, then $\irr(Y) \leq C_{\varepsilon}\cdot (n\cdot d)^{6+\varepsilon}$.
\item If $\gamma = 1$, then $\irr(Y)\leq C_{\varepsilon}\cdot (n\cdot d)^{7+\varepsilon}$. If furthermore $8\nmid n+1,8\nmid d$, then $\irr(Y) \leq C_{\varepsilon}\cdot (n\cdot d)^{6+\varepsilon}$.
\item If $\gamma=2$ and $4\nmid n+1, 4\nmid d$, then $\irr(Y) \leq C\cdot (n\cdot d)^{8}$.
\end{enumerate}
\end{thm}

\begin{proof}
The proof uses Lemma~\ref{lemma:extendabilityKumn} and is similar to the proof of Theorem~\ref{thm:irratK3n}. The only difference is the factor
$
    [\widetilde{\uO}^+(\Lambda_h)
    : \mathrm{Mon}^2(\Lambda,h)
    \cap\widetilde{\uO}^+(\Lambda_h)].
$
In the current case, $\mathrm{Mon}^2(\Lambda,h)$ is a subgroup of index at most two in~$\widehat{\uO}^+(\Lambda,h)$, so that $\mathrm{Mon}^2(\Lambda,h)\cap \widetilde{\uO}^+(\Lambda_h)$ is a subgroup of index at most two in
$
    \widehat{\uO}^+(\Lambda,h)
    \cap\widetilde{\uO}^+(\Lambda_h)
    = \widetilde{\uO}^+(\Lambda_h),
$
where the last equality follows from Lemma~\ref{lemma:monkumn}. 
\end{proof}

\begin{rmk}
Although \cite[Proposition 3.2]{GHS10}, used in the proof of Theorem~\ref{thm:irratK3n}, is stated for the K3$^{[n]}$-lattice, it applies equally in the Kum$_{n}$ case by replacing $n-1$ with $n+1$ since it can be interpreted as a statement about the rank two lattice $Q_h(-1)$ and its discriminant group. 
\end{rmk}

\subsection{Hyperk\"ahler manifolds of OG10-type}

Now consider hyperk\"ahler manifolds $X$ deformation equivalent to O'Grady's ten-dimensional example; see \cite{OGr99}. Then $H^2(X,\mathbb{Z})$ is isomorphic to the lattice $\Lambda = \Lambda_{\textrm{OG10}}:= U^{\oplus 3}\oplus E_8(-1)^{\oplus 2}\oplus A_2(-1)$; see \cite{Rap08}. Note that this is an even lattice of signature~$(3,21)$. The monodromy group $\operatorname{Mon}^2(\Lambda)$ coincides with the group $\widehat{\uO}^+(\Lambda) = \uO^+(\Lambda)$; see \cite{Ono22}. If $h \in \Lambda$ is a class with positive square $2d$, then its divisibility can only be $\gamma=1$ or $\gamma=3$. In this setting, the lattice $\Lambda_h$ is of the form
$$
    \Lambda_h
    \cong U^{\oplus 2}\oplus E_8(-1)^{\oplus 2} \oplus Q_{h}(-1), 
$$
where $Q_h(-1)$ is a certain negative-definite lattice of rank~$3$. Note that the lattice $\Lambda_h$ has signature $(2,21)$ and discriminant $|\operatorname{disc}(\Lambda_h)| = \frac{6d}{\gamma^2}$. The relevant monodromy group can be described explicitly as follows.

\begin{lemma}
\label{lemma:monOG10}
Consider the subgroup $\operatorname{Mon}^2(\Lambda,h) \subset \uO^+(\Lambda_h)$. 
\begin{enumerate}
\item\label{l:mOG10-1} If $\gamma=3$, then $\operatorname{Mon}^2(\Lambda,h) = \widetilde{\uO}^+(\Lambda_h)$. In this case, there exists a primitive embedding of $\Lambda_h$ into $\Lambda_{\#} = U^{\oplus 2}\oplus E_8(-1)^{\oplus 3}$ such that $\operatorname{Mon}^2(\Lambda,h)$ is extendable.
\item\label{l:mOG10-2} If $\gamma=1$, then $Q_h(-1)\cong A_2(-1)\oplus \mathbb Z\ell$ with $(\ell,\ell)=-2d$. In this case, the monodromy group has the form $\operatorname{Mon}^2(\Lambda,h) = \langle \widetilde{\uO}^+(\Lambda_h), \sigma_v \rangle$, where $\sigma_v$ is the reflection with respect to a vector $v\in A_2(-1)$ with $(v,v) \!= \!-6$. Moreover, there exists a primitive embedding of $\Lambda_h$ into $\Lambda_\# = U^{\oplus 2}\oplus E_8(-1)^{\oplus 2}\oplus A_2(-1)\oplus A_1(-1)^{\oplus 5}$ such that  $\operatorname{Mon}^2(\Lambda,h)$ is extendable. Furthermore, if $d$ is not divisible by $8$, then $A_1(-1)^{\oplus 5}$ can be replaced by $A_1(-1)^{\oplus 4}$.
\end{enumerate}
\end{lemma}

\begin{proof}
The formulas for $\operatorname{Mon}^2(\Lambda,h)$ and $Q_h(-1)$ in both cases are part of \cite[Theorem~3.1]{GHS11} and the proof of \cite[Lemma~4.4]{GHS11}. Let us find the lattice $\Lambda_\#$ and prove the extendability case by case.

\eqref{l:mOG10-1}~
Suppose  $\gamma = 3$. First note that $Q_h(-1)$ has signature $(0,3)$, so it can be embedded as a primitive sublattice of $E_8(-1)$ (see \cite[Theorem~14.1.15]{Huy16}). This induces a primitive embedding $\Lambda_h\hookrightarrow\Lambda_{\#} = U^{\oplus 2}\oplus E_8(-1)^{\oplus 3}$. Because $\operatorname{Mon}^2(\Lambda,h) = \widetilde{\uO}^+(\Lambda_h)$ in this case, it is extendable for every such embedding.
    
\eqref{l:mOG10-2}~
Suppose  $\gamma = 1$. Using Lagrange's four square theorem, we can write $d = a_1^2 + a_2^2 + a_3^2 + a_4^2 + 1$ for some integers $a_1,\dots,a_4$. This induces a primitive embedding of $\bZ\ell$ into $A_1(-1)^{\oplus 5}$ which maps a generator to $(a_1,\dots,a_4,1)$. This defines a primitive embedding of $Q_h(-1)$ into $A_2(-1)\oplus A_1(-1)^{\oplus 5}$, thus induces a primitive embedding of $\Lambda_h$ into $\Lambda_\# = U^{\oplus 2}\oplus E_8(-1)^{\oplus 2}\oplus A_2(-1)\oplus A_1(-1)^{\oplus 5}$. We fix a basis $\{\delta_1,\delta_2\}$ for $A_2(-1)$ and choose $v=\delta_1-\delta_2$ (see proof of \cite[Lemma~4.4]{GHS11}). We regard $v$ as an element in $\Lambda_{\#}$ and consider the corresponding reflection $\widetilde{\sigma}_v\in{\rm{O}}\left(\Lambda_\#\otimes\mathbb{Q}\right)$:
    \[
    \widetilde{\sigma}_v\colon x\longmapsto x-2\frac{(x,v)}{(v,v)}v.
    \]
    Note that $(v,v)=-6$ and for any $x\in\Lambda_{\#}$, $(x,v)$ is divisible by $3$. Then $\widetilde{\sigma}_v$ is integral, that is, $\widetilde{\sigma}_v\in {\rm{O}}^+\left(\Lambda_{\#}\right)$, and restricts to $\sigma_v$ on $\Lambda_h$. If furthermore $8\nmid d$, then \cite[Theorem 1]{CH07} shows that $d$ is the sum of four coprime squares so that there is a primitive embedding of $\mathbb{Z}\ell$ into $A_1(-1)^{\oplus 4}$.
\end{proof}

Let $\mathcal{M}_{\mathrm{OG10},\; 2d}^\gamma$ be the moduli space of hyperk\"{a}hler manifolds of type OG10 with a primitive polarization of degree $2d$ and divisibility $\gamma$. For both $\gamma=1$ and $\gamma=3$, this moduli space is irreducible; see \cite{Ono22b, Ono22} or \cite[Proposition~3.4]{Son22}. In the non-split ($\gamma=3$) case, $\mathcal{M}_{\mathrm{OG10},\; 2d}^\gamma$ is non-empty if and only if $2d\equiv 12\mod 18$; \textit{cf.} \cite[Lemma 3.4]{GHS11}. Moreover, all moduli spaces $\mathcal{M}_{\mathrm{OG10},\; 2d}^\gamma$ are of general type with finitely many exceptions \cite{GHS11,BBBF23}. We can now get a polynomial bound on their degrees of irrationality.

\begin{thm}
\label{thm:OG10_bound}
For any $\varepsilon>0$, there exists a constant $C_\varepsilon>0$  independent of $\gamma,d$ such that 
\[
\irr\left(\mathcal{M}_{\mathrm{OG10},\; 2d}^\gamma\right)\leq C_{\varepsilon} \cdot d^{14+\varepsilon}. 
\]
Furthermore, if $\gamma=1$, then this bound can be refined to $\irr\left(\mathcal{M}_{\mathrm{OG10},\; 2d}^1\right)\leq C_{\varepsilon} \cdot d^{\frac{27}{2}+\varepsilon}$ for all $d$ and to $\irr\left(\mathcal{M}_{\mathrm{OG10},\; 2d}^1\right)\leq C_{\varepsilon} \cdot d^{13+\varepsilon}$ for all $d$ not divisible by $8$.
\end{thm}

\begin{proof}
    One uses Lemma~\ref{lemma:monOG10}, proceeding as in the proof of Theorem~\ref{thm:irratK3n}, in the case $n=2$ or $\gamma\geq 3$.
    The only different step here is the estimate of $|\uO(D(\Lambda_h))|$. We need to prove that for any $\varepsilon>0$, there exists a $C'_{\varepsilon}>0$ such that 
    $|\uO(D(\Lambda_h))|<C'_\varepsilon\cdot d^{\varepsilon}$.
    
    If $\gamma=3$, then $2d=18t-6$ for some integer $t$. The discriminant group $D\left(\Lambda_h\right)\cong D\left(Q_h(-1)\right)$ is cyclic of order $2d/3$; see \cite[Theorem~3.1]{GHS11}. Further, in this case, $Q_h(-1)$ is given by 
    \[
    Q_h(-1)=\begin{pmatrix}
        -2 & 1 & 0 \\
        1 & -2 & -1\\
        0 & -1 & -2t
    \end{pmatrix}; 
    \]
    see \cite[Section 8]{BBBF23}. Let $\{\delta_1, \delta_2, \delta_3\}$ be a basis for $Q_h(-1)$ and $r=2d/3$. Then $\delta_1+2\delta_2-3\delta_3$ has divisibility $r$, and $\alpha=\frac{\delta_1+2\delta_2-3\delta_3}{r}+\Lambda_h$ is a generator for $D\left(Q_h(-1)\right)$. An element $g\in \uO(D(\Lambda_h))$ is given by~$\alpha\mapsto a\cdot\alpha$, with 
    \[
    (\alpha,\alpha)\cdot\left(a^2-1\right)=\frac{-9}{2d}(a^2-1)=\frac{-3}{r}(a^2-1)\in \mathbb{Z}.
    \]
    Note that $r\equiv -2\mod3$, and in particular  $a^{2}\equiv 1$ modulo $r$. Looking at the decomposition $r=p_1^{k_1}\cdots p_s^{k_s}$, where the $p_i$ are distinct primes, by the Chinese remainder theorem, we see that the number of solutions to~$x^2\equiv 1$~modulo~$r$ is bounded by $2^{\rho(r)+1}$, where $\rho(r)$ is the number of distinct prime factors of $r$. Hence, there are at most $2^{\rho(r)+1}\leq 2^{\rho(2d)+1}\leq 2^{\rho(d)+2}$ possible values for $a$ modulo $r$. Then we proceed as in the proof of Theorem~\ref{thm:irratK3n}.
    
    If instead $\gamma=1$, then we have $D(\Lambda_h)=D(\mathbb{Z}\ell)\oplus D(A_2(-1))$, where the direct sum is with respect to the $\mathbb{Q}/ 2\mathbb{Z}$-valued quadratic form. If $(\ell,e_1,e_2)$ is a canonical basis of $\mathbb{Z}\ell\oplus A_2(-1)$ , then $D(\Lambda_h)$ is generated by $\frac{\ell}{2d},\frac{2}{3}e_1+\frac{1}{3}e_2$ modulo $\Lambda_h$.
     An element $A\in \uO(D(\Lambda_h))$ is determined by the images of the generators, 
     \[ A\left( \frac{\ell}{2d} \right) = a_{11}\left( \frac{\ell}{2d} \right) + a_{21}\left( \frac{2}{3}e_1+\frac{1}{3}e_2 \right),\quad A\left( \frac{2}{3}e_1+\frac{1}{3}e_2 \right) = a_{12}\left( \frac{\ell}{2d} \right) + a_{22}\left( \frac{2}{3}e_1+\frac{1}{3}e_2 \right);  \]
     hence it can be represented by a $2\times 2$ matrix $A=(a_{ij})_{i,j=1,2}$ with  $a_{1j}\in\{0,1,\ldots,2d-1\}$ and $a_{2j}\in\{0,1,2\}$. The fact that $A$ preserves the quadratic form implies that
    \[
        A^t \begin{pmatrix}
        \frac{1}{2d} & 0 \\
        0 & \frac{2}{3}
    \end{pmatrix} A
    \;\equiv\;
    \begin{pmatrix}
        \frac{1}{2d} & 0 \\
        0 & \frac{2}{3}
    \end{pmatrix}
    \;\mod\; \mathbb{Z}, 
     \]
     which can be written down explicitly as
     \begin{equation}\label{eq:orthogonalityexplicit}
     \frac{1}{2d}a_{11}^2+\frac{2}{3}a_{21}^2 \equiv \frac{1}{2d},
     \quad \frac{1}{2d}a_{11}a_{12}+\frac{2}{3}a_{21}a_{22} \equiv 0,
     \quad \frac{1}{2d}a_{12}^2+\frac{2}{3}a_{22}^2 \equiv \frac{2}{3} \quad \mod \mathbb{Z}. 
     \end{equation}
     First suppose  $3\nmid d$. Then multiplying the equations of \eqref{eq:orthogonalityexplicit} by $6d$, we get 
     \[ 
     3\cdot a_{11}^2+4d\cdot  a_{21}^2 \equiv 3,
     \quad 3\cdot a_{11}a_{12}+4d\cdot a_{21}a_{22} \equiv 0,
     \quad 3\cdot a_{12}^2+4d\cdot a_{22}^2 \equiv 4d
     \quad \mod 6d\cdot\mathbb{Z}. 
     \]
     Looking at the last equations modulo $d$ and using that $3$ is coprime to $d$, we get
     \[
    a_{11}^2\equiv 1
    \quad\text{and}\quad
    a_{11}\cdot a_{12}\equiv 0\mod d.
     \]
      In particular, $a_{12} \equiv 0$ modulo $d$, so $a_{12}$ must belong to $\{0,d\}$. Again, since the number of solutions to $x^2\equiv 1$ modulo $d$ is bounded by $2^{\rho(d)+1}$, there are at most $4\cdot 2^{\rho(d)}$ possible values for $a_{11}$ modulo $2d$. Since $a_{21},a_{22}$ can take at most three possible values, all these facts yield $|\uO(D(\Lambda))| \leq 2\cdot (4\cdot 2^{\rho(d)})\cdot 3\cdot 3 = 72\cdot2^{\rho(d)}$. If instead $d=3r$, we can multiply Equations~\eqref{eq:orthogonalityexplicit} by $6r$ and use a similar reasoning. As a consequence, for any $\varepsilon>0$, we see that there exists a $C'_{\varepsilon}>0$ such that $\left|\uO(D(\Lambda_h))\right| \leq C'_{\varepsilon}\cdot  d^{\varepsilon}$.
\end{proof}

\subsection{Hyperk\"{a}hler manifolds of OG6-type}

Let $X$ be a hyperk\"ahler manifold of dimension $2n$ deformation equivalent to O'Grady's six-dimensional example; see \cite{OGr03}. Then the group $H^2(X,\mathbb{Z})$ is isomorphic to the lattice $\Lambda = \Lambda_{\rm OG6}:= U^{\oplus 3}\oplus A_1(-1)^{\oplus 2}$, which is an even lattice of signature $(3,5)$. In this case, we have that (\textit{cf.} \cite{MR21})
\[  \operatorname{Mon}^2(\Lambda)
    = {\uO}^+(\Lambda).
\]
Let $h \in \Lambda$ be a primitive class with $(h,h)=2d>0$. Then its divisibility is either $\gamma=1$ or $\gamma=2$, and the lattice $\Lambda_h$ has signature $(2,5)$ and discriminant $|\operatorname{disc}(\Lambda_h)| = \frac{8d}{\gamma^2}$. We will denote the standard bases of the first two copies of $U$ by $\{e,f\}$ and $\{e_1,f_1\}$ and the canonical basis of $A_1(-1)^{\oplus 2}$ by $\{v_1,v_2\}$.

\begin{lemma}
\label{lemma:OG6orbit}
Up to the action of\, $\operatorname{Mon}^2(\Lambda)$, the class $h$ and the orthogonal complement $\Lambda_h$ have the following form:
\begin{enumerate}
\item If\, $\gamma=1$, then $h=e+df$ and $\Lambda_h = U^{\oplus 2} \oplus \langle e-df \rangle \oplus A_1(-1)^{\oplus 2}$.
\item If\, $\gamma=2$, then $h=2(e+tf)-w$ for some $w\in\{v_1,v_2,v_1+v_2\}$ and some integer $t$. In this case, we have $\Lambda_h = U^{\oplus 2}\oplus Q_h(-1)$,  where
  \begin{equation*}
    Q_h(-1)=
    \begin{cases}
    \langle f-v_1,e-tf,v_2 \rangle & \text{if }\, w=v_1,\\
    \langle f-v_2,e-tf,v_1 \rangle & \text{if }\, w=v_2,\\
    \langle f-v_1,f-v_2,e-tf \rangle & \text{if }\, w=v_1+v_2.
    \end{cases}
  \end{equation*}
\end{enumerate}
\end{lemma}

\begin{proof}
Every element $h\in \Lambda$ can be written as $av-w$ with $v\in U^{\oplus 3}$ primitive and $w\in A_1(-1)^{\oplus 2}$. Since $h$ is primitive, by Eichler's criterion, we can assume $w\in\{0,v_1,v_2,v_1+v_2\}$. If $\gamma=1$, then $w=0$ and $a=1$. Then, using Eichler's criterion as in the proof of Lemma~\ref{lemma:kumnhorbit}, we see that $h$ can be taken as $e+df$.

If $\gamma=2$, then $w\in\{v_1,v_2,v_1+v_2\}$. Further, 
$2$ divides $a$ and since $\Lambda$ is even, $(v,v)\in 2\mathbb{Z}$. Then
\[
2d=(h,h)=a^2(v,v)+(w,w)=8t+(w,w) 
\]
for a certain integer $t$. Then 
\[
d=\begin{cases}4t-2&\hbox{ if }w=v_1+v_2,\\
4t-1&\hbox{ if }w\in \{v_1,v_2\}.\end{cases}
\] 
Using Eichler's criterion again, we can take $h=2(e+tf)-w$ with $w\in\{v_1,v_2,v_1+v_2\}$ for $\gamma=2$. Based on this, the lattice $\Lambda_h$ can be computed explicitly.
\end{proof}

The following lemma is here for the sake of completeness; \textit{cf.} \cite[Lemma 7.1]{BBFW24}.

\begin{lemma}
\label{lemma:monOG6}
The group ${\rm{Mon}}^2\left(\Lambda,h\right)$ is at most a degree two extension of\, $\widetilde{\rm{O}}^+\left(\Lambda,h\right)$. More concretely,
\[
{\rm{Mon}}^2\left(\Lambda,h\right)=\left\langle\widetilde{\rm{O}}^+\left(\Lambda,h\right), \sigma\right\rangle, 
\]
where $\sigma$ is the identity if $w\in\{v_1,v_2\}$, and $\sigma=\sigma_\kappa\in{\rm{O}}^+\left(\Lambda,h\right)$ is the reflection along $\kappa=v_1-v_2$ if $\gamma=1$ or $\gamma=2$ and $w=v_1+v_2$.
\end{lemma}

\begin{proof}
Since $\Lambda$ is indefinite of rank eight, the map ${\rm{O}}\left(\Lambda\right)\rightarrow{\rm{O}}\left(D\left(\Lambda\right)\right)$ is surjective; see \cite[Theorem~1.14.2]{Nik80}. Further, $D\left(\Lambda\right)\cong \left(\mathbb{Z}\big/2\mathbb{Z}\right)^{\oplus 2}$ is generated by $\{\frac{v_1}{2},\frac{v_2}{2}\}$. One notes that if $g\in{\rm{O}}\left(D\left(\Lambda\right)\right)$, then $g={\rm{Id}}$ or $g:\frac{v_1}{2}\mapsto\frac{v_2}{2}, \frac{v_2}{2}\mapsto\frac{v_1}{2}$. In particular,  
\[\left[{\rm{O}}^+\left(\Lambda,h\right):\widetilde{\rm{O}}^+\left(\Lambda,h\right)\right]\leq 2.\]
One can check that $\sigma_\kappa\in{\rm{O}}^+\left(\Lambda,h\right)$ fixes $h$ and does not act as the identity on $D(\Lambda)$. Assume $\gamma=1$ and $\kappa=v_1-v_2$. Since $h\in (v_1-v_2)^\perp$, see Lemma~\ref{lemma:OG6orbit}, then $\sigma_k(h)=h$, that is, $\sigma_k\in {\rm{O}}^+\left(\Lambda,h\right)$. Further, 
\[
\sigma_k\left(\frac{v_1}{2}\right)=\frac{v_1}{2}-2\frac{\left(\frac{v_1}{2},\kappa\right)}{(\kappa,\kappa)}\kappa=\frac{v_2}{2}.
\]
In particular, $\sigma_k\not\in \widetilde{\rm{O}}^+\left(\Lambda,h\right)$. The case $\gamma=2$ and $w=v_1+v_2$ is analogous. Finally, assume $\gamma=2$ and $w=v_1$. Let $g\in {\rm{O}}^+\left(\Lambda,h\right)$. Then $g(h)=h$ and $g\left(\frac{h}{2} \right) = \frac{h}{2}$.
In particular, from Lemma~\ref{lemma:OG6orbit} we have $g\left(\frac{h}{2}\right)=g(e+tf)-g\left(\frac{v_1}{2}\right)= (e+tf)-\frac{v_1}{2}$, which implies $g\left(\frac{v_1}{2}\right)\equiv \frac{v_1}{2}\mod \Lambda$. Therefore, there is no element in ${\rm{O}}^+\left(\Lambda\right)$ fixing $h$ and acting as $\frac{v_1}{2}\mapsto \frac{v_2}{2}$ on the discriminant of $\Lambda$. In particular, every such $g$ acts as the identity on $D(\Lambda)$. This shows the equality ${\rm{O}}^+\left(\Lambda,h\right)=\widetilde{\rm{O}}^+\left(\Lambda,h\right)$. The argument is the same if $w=v_2$.
\end{proof}

\begin{lemma}
\label{lemma:exOG6}
We have the equality
$
    \widetilde{\uO}^+(\Lambda,h)
    = \widetilde{\uO}^+(\Lambda_h).
$
\end{lemma}

\begin{proof}
The inclusion $\widetilde{\uO}^+(\Lambda_h)\subset \widetilde{\uO}^+(\Lambda,h)$ follows from Lemma~\ref{lem:O+tilde}. To prove the converse, let us first show that if $x\in\Lambda_h^{\vee}$, then $x\in\bQ\cdot h + \Lambda^{\vee}$; since $\Lambda_h\subset\Lambda^{\vee}$, it suffices to check this for those $x\in\Lambda_h^{\vee}$ whose classes generate $D(\Lambda_h)$. Let us proceed case by case:

\begin{enumerate}
\item First assume  $\gamma=1$. Then in the notation of Lemma~\ref{lemma:OG6orbit}, one can compute that
\[
    D(\Lambda_h) = \left\langle
        \frac{1}{2d}(e-df), \frac{1}{2}v_1,\frac{1}{2}v_2
    \right\rangle. 
\]
Observe that $\frac{1}{2}v_1,\frac{1}{2}v_2\in \Lambda^{\vee}$ and $\frac{1}{2d}(e-df) = \frac{1}{2d}(e+df)-f = \frac{1}{2d}h-f$ with $f\in \Lambda$.

\item Now assume $\gamma=2$. Due to Lemma~\ref{lemma:OG6orbit}, we can assume $h = 2(e+tf)-\omega$ with $\omega\in\{v_1,v_2,v_1+v_2\}$. If $\omega=v_1$, then one can verify using the same lemma that
\[
    D(\Lambda_h) = D(Q_h(-1)) = \left\langle
        \frac{1}{2d}h - \frac{1}{2}v_1,
        \frac{1}{d}h - f,
        \frac{1}{2}v_2
    \right\rangle = \left\langle
        \frac{1}{2d}h - \frac{1}{2}v_1,
        \frac{1}{2}v_2
    \right\rangle, 
\]
where the third equality comes from the relation $\frac{1}{d}h-f = 2\left(\frac{1}{2d}h-\frac{1}{2}v_1\right)- (f-v_1)$ and the fact that $f-v_1\in Q_h(-1)$. Observe that the generators are all contained in $\mathbb{Q}\cdot h+\Lambda^{\vee}$. The case $\omega=v_2$ can be solved in a similar way. If instead $\omega=v_1+v_2$, then one can verify that
\begin{align*}
    D(\Lambda_h) = D(Q_h(-1))
    &= \left\langle
        \frac{1}{2d}h - \frac{1}{2}v_1,
        \frac{1}{2d}h - \frac{1}{2}v_2,
        \frac{1}{d}h - f
    \right\rangle \\
    &= \left\langle
        \frac{1}{2d}h - \frac{1}{2}v_1,
        \frac{1}{2d}h - \frac{1}{2}v_2
    \right\rangle.
\end{align*}
Observe that these generators are contained in $\mathbb{Q}\cdot h+\Lambda^{\vee}$.
\end{enumerate}
This completes the proof for the inclusion $\Lambda_h^{\vee}\subset\bQ\cdot h + \Lambda^{\vee}$.
Now, pick any $g\in\uO(\Lambda,h)\cap \widetilde{\uO}^+(\Lambda)$ and $x\in\Lambda_h^{\vee}$. Then $x\in\mathbb{Q}\cdot h + \Lambda^{\vee}$, which allows us to write $x=\alpha \cdot h + y$ with $\alpha\in\mathbb{Q}$ and $y\in\Lambda^{\vee}$. It follows that
\[
    g(x) - x
    = g(\alpha\cdot h) + g(y) - \alpha\cdot h - y
    = g(y) - y \in \Lambda.
\]
We also have
$$
    (g(x)-x, h)
    = (g(x), h) - (x, h)
    = (x, g^{-1}(h)) - (x, h)
    = (x, h) - (x, h)
    = 0.
$$
These two relations imply that $g(x) - x\in\Lambda_h$. Hence $g\in\widetilde{\uO}^+(\Lambda_h)$. This finishes the proof of the equality 
$
    \widetilde{\uO}^+(\Lambda_h)
    = \uO(\Lambda,h)\cap\widetilde{\uO}^+(\Lambda).
$
\end{proof}

As a consequence, we have the following extendability result. 

\begin{cor}
\label{sec3:cor:extMon}
In all cases listed in Lemma~\ref{lemma:OG6orbit}, there is a primitive embedding $\Lambda_h\hookrightarrow\Lambda_{\#}$ such that the group ${\rm{Mon}}^2\left(\Lambda,h\right)$ extends. 
\begin{enumerate}
\item\label{sec3:c:eMon-1} If\, $\gamma=1$, then $\Lambda_{\#} = U^{\oplus 2}\oplus A_{1}(-1)^{\oplus 7}$, and if $d$ is not divisible by $8$, then $A_{1}(-1)^{\oplus 7}$ can be replaced by $A_{1}(-1)^{\oplus 6}$.
\item\label{sec3:c:eMon-2}  If\, $\gamma=2$ and $w\in \{v_1,v_2\}$, then $\Lambda_{\#} = U^{\oplus 2}\oplus E_8(-1)$.
\item\label{sec3:c:eMon-3}  If\, $\gamma=2$ and $w=v_1+v_2$, then $\Lambda_{\#} =U^{\oplus 2}\oplus A_3(-1)\oplus A_1(-1)^{\oplus 4}$.

\end{enumerate}
\end{cor}

\begin{proof}
By Lemmas~\ref{lemma:monOG6} and~\ref{lemma:exOG6}, it suffices to show that $\sigma_\kappa$ extends to $\Lambda_{\#}$ is all cases. We proceed case by case:

\eqref{sec3:c:eMon-1}~
By Lemma~\ref{lemma:OG6orbit}, $\Lambda_h = U^{\oplus 2} \oplus \langle e-df \rangle \oplus A_1(-1)^{\oplus 2}$ and, as in the proof of Lemma~\ref{lemma:monOG10}, Lagrange's four square theorem induces a primitive embedding of $\langle e-df \rangle$ in $A_1(-1)^{\oplus 5}$ (resp.\ $A_1(-1)^{\oplus 4}$ if $d$ is not divisible by $8$). This gives us an embedding 
\[
\Lambda_h=U^{\oplus 2} \oplus \langle e-df \rangle \oplus A_1(-1)^{\oplus 2}\longhookrightarrow U^{\oplus 2} \oplus A_1(-1)^{\oplus 5} \oplus A_1(-1)^{\oplus 2}.
\]
Further, note that if we see $\kappa$ as an element in $\Lambda_{\#}$ and we call $\widetilde{\sigma}_\kappa\in{\rm{O}}^+\left(\Lambda_{\#}\otimes\mathbb{Q}\right)$ the corresponding reflection 
\begin{equation}
\label{sec3:eq:ref}
\widetilde{\sigma}_\kappa\colon v\longmapsto v-2\frac{(v,\kappa)}{(\kappa,\kappa)}\kappa,
\end{equation}
then since $(\kappa,\kappa)=-4$ and $(v,\kappa)$ is divisible by $2$ for all $v\in \Lambda_{\#}$, the reflection is integral, that is, $\widetilde{\sigma}_\kappa\in {\rm{O}}^+\left(\Lambda_{\#}\right)$. Since $\widetilde{\sigma}_\kappa\mid_{\Lambda_h}=\sigma_\kappa$, this shows the  extendability in the first case.

\eqref{sec3:c:eMon-2}~
Note that $\ell(\Lambda_h)\leq 3$; see the proof of Lemma~\ref{lemma:exOG6}. This fact, together with \cite[Theorem~14.1.15]{Huy16}, shows that in this case, there exists a primitive embedding $\Lambda_h \hookrightarrow U^{\oplus 2}\oplus E_{8}(-1)$. The extendability of ${\rm{Mon}}^2\left(\Lambda,h\right)$ follows from Lemmas~\ref{lemma:monOG6} and~\ref{lemma:exOG6}, and from the fact that $\widetilde{\rm O}^+(\Lambda_h)$ is always extendable.

\eqref{sec3:c:eMon-3}~  
Let $\left\{\delta_1,\delta_2,\delta_3\right\}$ be a basis of $A_3(-1)$ and $\{e_1,\ldots,e_4\}$ a basis of $A_1(-1)^{\oplus 4}$. By Lagrange's four square theorem, there are integers $a_1,\ldots,a_4$ such that $t-2=a_1^2+\cdots+a_4^2$.
In particular, there is a primitive embedding $Q_h(-1)\hookrightarrow A_3(-1)\oplus A_1(-1)^{\oplus 4}$ given by 
\[
f-v_1\longmapsto \delta_1,\quad f-v_2\longmapsto \delta_3,\quad e-tf\longmapsto a_1e_1+\cdots+a_4e_4-\delta_1-\delta_2-\delta_3.
\]
Further, from Lemma~\ref{lemma:monOG6}, $\kappa\in \Lambda_h\subset \Lambda_{\#}$ is given by $\kappa=\delta_3-\delta_1$. One notes that $(\kappa,\kappa)=-4$ and $(v,\kappa)$ is even for all $v\in \Lambda_{\#}$. In particular, the reflection \eqref{sec3:eq:ref} is integral and restricts to $\sigma_\kappa$ on $\Lambda_{k}$. This shows the extendability in the third case.
\end{proof}

Let $\mathcal{M}_{\mathrm{OG6},\; 2d}^\gamma$ be the moduli space of hyperk\"{a}hler manifolds of type OG6 and polarized by a primitive ample class of degree $2d$ and divisibility $\gamma$. By \cite[Propositions~3.4 and~3.6]{Son22}, the space $\mathcal{M}_{\mathrm{OG6},\; 2d}^\gamma$ is irreducible. Moreover, the proof of Lemma~\ref{lemma:OG6orbit} shows that if the moduli space is non-empty, then either $\gamma=1$, or $\gamma=2$ and $d\equiv 2,3\mod 4$.

\begin{thm}
\label{sec3:them:OG6}
For any $\varepsilon>0$, there exists a $C_\varepsilon>0$ independent of $d$ and $\gamma$ such that 
$$
    \irr\left(
        \mathcal{M}_{\mathrm{OG6},\; 2d}^\gamma
    \right)
    \leq C_\varepsilon\cdot d^{6+\varepsilon}.
$$
If instead $\gamma=1$ or $\gamma=2$ and $d\equiv 3\mod 4$, the bound can be refined to
\[
 \irr\left(
        \mathcal{M}_{\mathrm{OG6},\; 2d}^\gamma
    \right)
    \leq C_\varepsilon\cdot d^{\frac{11}{2}+\varepsilon}.
\]
Finally, if $\gamma=1$ and $d$ is not divisible by $8$,
the bound can be further refined to
\[
 \irr\left(
        \mathcal{M}_{\mathrm{OG6},\; 2d}^\gamma
    \right)
    \leq C_\varepsilon\cdot d^{5+\varepsilon}.
\]

\end{thm}

\begin{proof}
The proof uses Corollary~\ref{sec3:cor:extMon} and is analogous to the proof of Theorem~\ref{thm:irratK3n}. The only difference is the estimate of $|\uO(D(\Lambda_h))|$. We want to prove that for any $\varepsilon>0$, there exists a $C'_{\varepsilon}>0$ such that $|\uO(D(\Lambda_h))| \leq C'_{\varepsilon}\cdot d^{\varepsilon}$. 

In the case $\gamma=1$, we can use the description of $D(\Lambda_h)$ given in the proof of Lemma~\ref{lemma:exOG6} and then proceed as in the proof of Theorem~\ref{thm:OG10_bound}: An element $A\in {\rm{O}}\left(D\left(\Lambda_h\right)\right)$ is determined by what it does to the generators $\left\{\frac{1}{2d}(e-df),\frac{1}{2}v_1,\frac{1}{2}v_2\right\}$. In particular, it can be represented by a $3\times3$ matrix $A=(a_{ij})$ with $a_{1j}\in \{0,1,\ldots,2d-1\}$ and $a_{2j}, a_{3j}\in\{0,1,2\}$. Preserving the quadratic form translates into
\[
        A^t \begin{pmatrix}
        \frac{1}{2d} & 0 & 0\\
        0 & \frac{1}{2} & 0\\
        0 & 0 & \frac{1}{2}
    \end{pmatrix} A
    \;\equiv\;
    \begin{pmatrix}
        \frac{1}{2d} & 0 & 0\\
        0 & \frac{1}{2} & 0\\
        0 & 0 & \frac{1}{2}
    \end{pmatrix}
    \;\mod\; \mathbb{Z}.
     \]
Multiplying the resulting six equations by $2d$ and then reducing modulo $d$, one gets
\[
a_{11}^2\equiv 1, \;\; a_{11}\cdot a_{12}\equiv a_{11}\cdot a_{13}\equiv a_{12}^2\equiv a_{12}\cdot a_{13}\equiv a_{13}^2\equiv 0\;\;\mod d.
\]
In particular, $a_{12}, a_{13}\in \{0,d\}$ and, since the number of solutions of $x^2\equiv 1\mod d$ is in ${\rm{O}}(d^\varepsilon)$, we obtain the desired bound for $\left|{\rm{O}}\left(D(\Lambda_h)\right)\right|$ in this case.

Suppose instead that $\gamma = 2$ and $d \equiv 3 \mod 4$. Then we can assume that $h=2(e+tf)-w$ as in Lemma~\ref{lemma:OG6orbit}, with $w\in\{v_1,v_2\}$. We can then use the description of the discriminant group given in the proof of Lemma~\ref{lemma:exOG6} and reason as in the proof of Theorem~\ref{thm:OG10_bound}. 

In the last case of $\gamma=2$ and $d\equiv 2 \mod 4$,  we can assume that $h=2(e+tf)-w$ with $w=v_1+v_2$. The proof of Lemma~\ref{lemma:exOG6} shows that in this case
\[
D(\Lambda_h) = \left\langle \frac{1}{2d}h-\frac{1}{2}v_1,\frac{1}{2d}h-\frac{1}{2}v_2\right\rangle = \left\langle \frac{1}{2d}h-\frac{1}{2}v_1,\frac{1}{2}v_2-\frac{1}{2}v_1\right\rangle. 
\]
Using this latter set of generators, we can proceed as in the proof of Theorem~\ref{thm:OG10_bound} and compute that an element of $\uO(D(\Lambda_h))$ corresponds to a $2\times 2$ matrix $A=(a_{ij})$ with $a_{1j}\in \{0,\dots,2d-1\}$ and $a_{2j}\in \{0,1\}$ subject to the orthogonality condition 
\[
        A^t \begin{pmatrix}
        \frac{d+1}{2d} & \frac{1}{2} \\[.6ex]
        \frac{1}{2} & 0 
    \end{pmatrix} A
    \;\equiv\;
    \begin{pmatrix}
        \frac{d+1}{2d} & \frac{1}{2} \\[.6ex]
        \frac{1}{2} & 0 
    \end{pmatrix}
    \;\mod\; \mathbb{Z}.
     \]
In particular, the entries must satisfy
\[ a_{11}^2 \equiv 1, \quad a_{11}\cdot a_{12} \equiv 0 \quad \mod d, \]
and we conclude as in the proof of Theorem~\ref{thm:OG10_bound}.
\end{proof}

\section{Moduli spaces of abelian surfaces and K3 surfaces}
\label{sec:Ab_surf}

This section consists of two parts. First, we study the irrationality of moduli spaces $\mathcal{A}_{(1,d)}$ of $(1,d)$-polarized abelian surfaces. Then we revisit our study in \cite{ABL23} about the irrationality of moduli spaces $\cF_d$ of primitively polarized K3 surfaces.

\subsection{Abelian surfaces of type (1, d)}

Our strategy in bounding the degree of irrationality of $\mathcal{A}_{(1,d)}$ starts by realizing it as a double cover of an appropriate period space using the construction of Gritsenko and Hulek \cite{GH98}. The construction will provide a bound for the degree of irrationality when $d$ is square-free. Our main task here is to reduce the general case to this one via a geometric argument by O'Grady; see \cite{OGr89}.

Let $\mathbb{H}_2$ be the Siegel upper-half space of $2\times 2$ symmetric complex matrices $\tau$ with positive-definite imaginary part. Consider the usual action of the symplectic group $\operatorname{Sp}(4,\mathbb{Q})$ on $\mathbb{H}_2$ and the arithmetic group
\[
    \Gamma_{1,d} = \left\{
    \begin{pmatrix}
        \mathbb{Z}&\mathbb{Z}&\mathbb{Z}&d\mathbb{Z} \\ d\mathbb{Z}&\mathbb{Z}&d\mathbb{Z}&d\mathbb{Z} \\ \mathbb{Z}&\mathbb{Z}&\mathbb{Z}&d\mathbb{Z}\\ \mathbb{Z}&\frac{1}{d}\mathbb{Z}&\mathbb{Z}&\mathbb{Z}
    \end{pmatrix}
    \right\}
    \cap{\rm{Sp}}\left(4,\mathbb{Q}\right).
\]
Then we have $\mathcal{A}_{(1,d)} = \mathbb{H}_2/\Gamma_{1,d}$. Consider the Euclidean lattice $L=\mathbb{Z}e_1\oplus \cdots\oplus \mathbb{Z}e_4$, where $(e_i,e_i) = 1$ and $(e_i,e_j) = 0$ for $i\neq j$. Then the space of bivectors $L\wedge L$ forms a rank six lattice, where the pairing $(x,y)\in \mathbb{Z}$ is defined by requiring that
\[
    x\wedge y
    = (x,y)\cdot (e_1\wedge e_2\wedge e_3\wedge e_4)
    \;\in\;
    \bigwedge^{4}L.
\]
Fix $w_d:= e_1\wedge e_3+de_2\wedge e_4\in L\wedge L$ and consider the \emph{integral paramodular group of level $d$}
\[
    \widetilde{\Gamma}_{1,d} = \{ g\in \mathrm{GL}(L) \,|\, (\wedge^2 g)(w_d)=w_d \}.
\]
Then the orthogonal complement $\Lambda_d = \left(w_d\right)^\perp\subset L\wedge L$ has the form
$$
    \Lambda_d\cong U^{\oplus 2}\oplus \mathbb{Z}\ell, 
    \quad\text{where }
    (\ell,\ell) = 2d.
$$
The group $\widetilde{\Gamma}_{1,d}$ acts naturally on the lattice $\Lambda_d$ via the group homomorphism
\[
    \widetilde{\Gamma}_{1,d}
    \longrightarrow \uO(\Lambda_d),
    \quad
    g \longmapsto (\wedge^2 g)|_{\Lambda_d}.
\]
It turns out that 
$
    \Gamma_{1,d}
    = \mathbb{I}_{d}^{-1}\widetilde{\Gamma}_{1,d}\mathbb{I}_d
$, 
where $\mathbb{I}_d = \operatorname{diag}(1,1,1,d)$. In particular, the group $\Gamma_{1,d}$ acts on $\Lambda_d$ via the homomorphism
\begin{equation}
\label{sec7:eq:Gamma_Otilde}
    \Gamma_{1,d}\longrightarrow\uO(\Lambda_d),
    \quad
    g\longmapsto\wedge^2\left(\mathbb{I}_d g \mathbb{I}_d^{-1}\right).  
\end{equation}

It is proved in \cite[Lemma~1.1]{GH98} that the image of this map lies inside $\widetilde{\uO}\left(\Lambda_d\right)$. In the case that $d$ is square-free, one of the key results in \cite{GH98} asserts the existence of an index two extension 
\[
    \Gamma_{1,d}\subset\Gamma_{1,d}^+
    \quad\text{together with a map}\quad
    \Gamma_{1,d}^+\longrightarrow \widetilde{\uO}^+\left(\Lambda_d\right)
\]
which induces a birational morphism from $\mathcal{A}_{(1,d)}^+\colonequals\mathbb{H}_2\big/\Gamma_{1,d}^+$ to the quotient $\Omega(\Lambda_d)/\widetilde{\uO}^+(\Lambda_d)$. This implies the following statement.

\begin{prop}[\textit{cf.} \protect{\cite[Proposition 1.4]{GH98}}]
\label{prop:absurfGH}
Assume that $d$ is square-free. Then there exists a dominant morphism of degree two
\[
    \mathcal{A}_{(1,d)}
    \longrightarrow\widetilde{\cP}^+_{\Lambda_d}
    = \Omega(\Lambda_d)\big/\widetilde{\uO}^+\left(\Lambda_d\right).
\]
\end{prop}

This allows us to bound the degree of irrationality of $\mathcal{A}_{(1,d)}$ using $\widetilde{\cP}^+_{\Lambda_d}$.

\begin{prop}
\label{prop:boundabeliansquarefree}
Suppose that $d$ is square-free. Then, for every $\varepsilon>0$, there exists a constant $C_\varepsilon>0$ independent of $d$ such that
\[
    \irr\left(\mathcal{A}_{(1,d)}\right)
    \leq C_\varepsilon\cdot d^{4+\varepsilon}.
\]
\end{prop}

\begin{proof}
From Proposition~\ref{prop:absurfGH}, we see that
$
    \irr(\mathcal{A}_{(1,d)})
    \leq 2\irr(\widetilde{\cP}_{\Lambda_d}).
$
Let us bound the latter using Lemma~\ref{lemma:irrationality-general}. Recall that $\Lambda_d \cong U^{\oplus 2}\oplus\bZ(-2d)$. Since $d$ is square-free, it is not divisible by $8$, and thus can be written as a sum of four coprime squares $d=\sum_{i=1}^4 a_i^2$. This induces a primitive embedding of $\Lambda_d$ into $\Lambda_{\#} = U^{\oplus 2} \oplus A_1(-1)^{\oplus 4}$. As $\widetilde{\uO}^+(\Lambda_d)$ is always extendable, Lemma~\ref{lemma:irrationality-general} shows that there exists a constant $C>0$ such that
\[
    \irr(\widetilde{\mathcal{P}}^+_{\Lambda_d})
    \leq C \cdot |\uO(D(\Lambda_d))| \cdot |\operatorname{disc} \Lambda_d|^4
    = 2^4\cdot C \cdot |\uO(D(\Lambda_d))| \cdot d^4.
\]
Because $D(\Lambda_d)\cong \mathbb{Z}/2d\mathbb{Z}$, we see that for every $\varepsilon>0$, there exists a constant $C'_{\varepsilon}>0$ such that we have $|\uO(D(\Lambda_d))| \leq C'_{\varepsilon}\cdot d^{\varepsilon}$. Putting all the inequalities together yields the desired bound.
\end{proof}

In order to extend this result to the general case, let us recall a geometric construction due to O'Grady; see \cite{OGr89}. Let $(A,L) \in \mathcal{A}_{(1,n^2k)}$ be an abelian surface with a polarization of type~$(1,n^2k)$. Then the natural map $\phi_L\colon A \to A^{\vee} = \operatorname{Pic}^0(A)$ satisfies $\operatorname{Ker}\phi_L \cong \left(\bZ/n^2k\bZ\right)^2$, and the subgroup of $n$-torsion points $J\colonequals\operatorname{Ker}\phi_L[n]$ is isomorphic to $\left(\mathbb{Z}/n\mathbb{Z}\right)^2$. Now consider the quotient $f\colon A \to B=A/J$. Then $B$ has a natural polarization $M$ of type $(1,k)$ which satisfies $f^*M=L$. This yields a map
\begin{equation}
\label{eq:mapogrady}
    \mathcal{A}_{(1,n^2k)}
    \longrightarrow \mathcal{A}_{(1,k)}.
\end{equation}

\begin{lemma}
\label{lemma:boundmapogrady}
For any $\varepsilon>0$, there exists a constant $C_{\varepsilon}$ independent of $n$ and $k$ such that the degree of this map is at most $C_{\varepsilon}\cdot (nk)^{4+\varepsilon}$. 
\end{lemma}

\begin{proof}
Let $\phi_M\colon B \rightarrow B^{\vee}$ be the isogeny induced by $M$. As in the proof of \cite[Proposition~5.1]{OGr89}, the image $H = f(\operatorname{Ker} \phi_L)$ is a subgroup of $B[nk]$ isomorphic to $\left(\mathbb{Z}/nk\mathbb{Z}\right)^2$ such that $\phi_M(H)$ equals the kernel of $f^*\colon B^{\vee} \to A^{\vee}$. In particular, since $f\colon A \to B$ is the dual of $f^*\colon B^{\vee} \to A^{\vee} = B^{\vee}/\phi_M(H)$, we see that we can reconstruct $A$ from $\phi_M$ and $H$. This shows that the cardinality of the fiber of \eqref{eq:mapogrady} over $(B,M)$ is bounded by the number of subgroups $H \subset B[nk]$ isomorphic to $(\mathbb{Z}/nk\mathbb{Z})^2$. Since $B[nk] \cong (\mathbb{Z}/nk\mathbb{Z})^4$, this number can be explicitly computed: If $nk$ is a prime power $p^r$, then \cite[Equation (1)]{But87} and \cite[Theorem 8.1]{Bir35} show that it equals $p^{4r-4}\cdot (p^2+1)\cdot (p^2+p+1) \leq C \cdot p^{4r}$ for a certain $C>1$ independent of $p$. If we write $nk=p_1^{r_1}\cdots p_s^{r_s}$ for $s=\rho(nk)$ as a product of pairwise distinct primes, then the Chinese remainder theorem implies that the number of subgroups is bounded by $C^{\rho(nk)}\cdot (nk)^4$, and $C^{\rho(nk)} = {\rm O}((nk)^{\varepsilon})$ for all $\varepsilon>0$. Indeed, if $C\leq 2^N$ for fixed $N$, then $C^{\rho(nk)}\leq \left(2^{\rho(nk)}\right)^N={\rm O}((nk)^{N\cdot\varepsilon})={\rm O}((nk)^{\varepsilon})$. 
\end{proof}

With this we can get a universal bound. For a positive integer $d$ factored into primes as $d=p_1^{2h_1+r_1} \cdots p_s^{2h_s+r_s}$, where $h_i\geq 0$ and $r_i\in \{0,1\}$, we then  define 
\[
k(d)=p_1^{r_1}\cdots p_s^{r_s}.
\]

\begin{thm}
\label{thm:boundabsurf}
For any $\varepsilon>0$, there exists a constant $C_{\varepsilon}>0$ independent of $d$ such that
\[
    \irr(\mathcal{A}_{(1,d)})
    \leq C_{\varepsilon}\cdot \left(d^{2+\varepsilon}\cdot k(d)^{6+\varepsilon}\right). 
\]
In particular, since $k(d)\leq d$, a universal polynomial bound is given by
\[
    \irr(\mathcal{A}_{(1,d)})
    \leq C_{\varepsilon}\cdot d^{8+2\varepsilon}.
\]
\end{thm}

\begin{proof}
Applying Lemma~\ref{lemma:boundmapogrady} with $n=\left(d/k(d)\right)^{\frac{1}{2}}$ and $k=k(d)$, we get a map
\[
    \mathcal{A}_{(1,d)}
    \longrightarrow
    \mathcal{A}_{(1,k(d))}
\]
whose degree is at most $C'_{\varepsilon}\cdot\left(n\cdot k(d)\right)^{4+\varepsilon}=C'_{\varepsilon}\cdot (d\cdot k(d))^{2+\frac{1}{2}\varepsilon}$.
Since $k(d)$ is square-free, Proposition~\ref{prop:boundabeliansquarefree} gives for each $\varepsilon>0$ a constant $C''_{\varepsilon}>0$ such that
\[
    \irr\left(\mathcal{A}_{(1,k(d))}\right)
    \leq C''_{\varepsilon} \cdot k(d)^{4+\varepsilon}.
\]
Hence, with $C_\varepsilon=C'_{\varepsilon}\cdot C''_{\varepsilon}$, we see that for all $\varepsilon>0$, there exists a constant $C_{\varepsilon}>0$ independent of $d$ such that 
\[
    \irr\left(\mathcal{A}_{(1,d)} \right)
    \leq C_{\varepsilon} \cdot \left(d^{2+\frac{1}{2}\varepsilon}\cdot k(d)^{6+\frac{3}{2}\varepsilon}\right)\leq C_{\varepsilon}\cdot d^{8+2\varepsilon}.
\]
This completes the proof.
\end{proof}

One can improve the above bound for infinitely many series of $d$ as in the case of K3 surfaces treated in \cite{ABL23}. One instance of this is when $k(d)$ is fixed, \textit{e.g.}~when $d$ is a square. We can also consider quadratic series on $d$. Consider positive integers $a,b,c\in \mathbb{Z}$ such that $4ac-b^2<0$. This defines a negative-definite rank two lattice
\begin{equation}
\label{eq:latticeranktwo}
    Q(-1)\colonequals\begin{pmatrix}
        -2a & b \\
        b & -2c
    \end{pmatrix}.
\end{equation}

\begin{prop}
\label{sec7:thm:d2}
Assume $d$ is a perfect square. Then for any $\varepsilon>0$, there exists a constant $C_{\varepsilon}>0$ independent of $d$ such that
\[    \irr\left(\mathcal{A}_{(1,d)}\right)
    \leq C_{\varepsilon}\cdot d^{2+\varepsilon}.
\]
Further, fix $a,b,c\in\mathbb{Z}$ which satisfy $4ac-b^2 < 0$. Suppose that $d$ is square-free and that it is of the form $d = aX^2-bXY+cY^2$. Then for any $\varepsilon>0$, there exists a constant $C_\varepsilon = C_\varepsilon(a,b,c) > 0$ such that
\[
    \irr\left(\mathcal{A}_{(1,d)}\right)
    \leq C_\varepsilon\cdot d^{2+\varepsilon}.
\]
\end{prop}

\begin{proof}
The first bound follows from Theorem~\ref{thm:boundabsurf} with $k(d)=1$. For the second bound, the assumption on $d$ implies that there exist a primitive embedding of $\mathbb{Z}(-2d)$ into $Q(-1)$ and consequently a primitive embedding of $U^{\oplus 2}\oplus\mathbb{Z}(-2d)$ into $U^{\oplus 2}\oplus Q(-1)$. Then the proof proceeds as in Proposition~\ref{prop:boundabeliansquarefree}.
\end{proof}

\subsection{Revisiting the case of K3 surfaces}

Let $\cF_{2d}$ be the moduli space of K3 surfaces with a primitive polarization of degree~$2d$. Recall that $\cF_{2d}$ is birational to
$
    \widetilde{\mathcal{P}}^+_{\Lambda_d}
    = \Omega(\Lambda_d)/\widetilde{\uO}^+(\Lambda_d)
$, 
where $\Lambda_d$ is the lattice
\[
    \Lambda_d\colonequals
    U^{\oplus 2}\oplus E_8(-1)^{\oplus 2}\oplus\mathbb{Z}\ell,
    \qquad (\ell,\ell) = -2d.
\]
Let $Q(-1)$ be a negative-definite lattice. Then every primitive embedding of $\mathbb{Z}(-2d)$ into $Q(-1)$ induces a primitive embedding of $\Lambda_d$ into $U^{\oplus 2}\oplus E_8(-1)^{\oplus 2} \oplus Q(-1)$. Since the stable orthogonal group $\widetilde{\uO}^+(\Lambda_d)$ is extendable with respect to such an embedding, we can use Lemma~\ref{lemma:irrationality-general} to bound $\irr(\cF_{2d})$. The general bound \cite[Theorem 1.1]{ABL23} is obtained by taking $Q(-1)=E_{8}(-1)$. The bounds \cite[Theorem 1.2]{ABL23} are obtained by choosing $Q(-1)=A_2(-1)$ in the case of associated special cubic fourfolds, $Q(-1)=A_1(-1)^{\oplus 2}$ in the case of associated special Gushel--Mukai fourfolds, and
\[
    Q(-1) = \begin{pmatrix}
        -2 & 0 \\
        0 & -2n
    \end{pmatrix}
    \quad\text{or}\quad
    \begin{pmatrix}
        -2 & -1 \\
        -1 & -\frac{n+1}{2}
    \end{pmatrix}
\]
in the case of associated special hyperk\"{a}hler fourfolds. We can get bounds for various series of $d$ by choosing suitable $Q(-1)$.

\begin{prop}
For any $\varepsilon>0$, there exists a constant $C_{\varepsilon}>0$ such that for any $d$ not divisible by $8$, one has
$$
    \irr\left(\mathcal{F}_{2d}\right)
    \leq C_{\varepsilon}\cdot d^{12+\varepsilon}.
$$
If\, $d$ is not congruent to $0, 4, 7\mod 8$, then
\[
    \irr\left(\mathcal{F}_{2d}\right)
    \leq C_{\varepsilon}\cdot d^{\frac{23}{2}+\varepsilon}.
\]
If\, $a,b,c$ are positive integers such that $4ac-b^2<0$, then
\[
{\rm{irr}}\left(\mathcal{F}_{2d}\right)\leq C_{\varepsilon}\cdot d^{10+\varepsilon}
\]
for all $d$ of the form $d=aX^2-bXY+cY^2$ with $X$ and $Y$ coprime.
\end{prop}

\begin{proof}
Since $d$ is not divisible by $8$, it can be written as a sum of four coprime squares $d = \sum_{i=1}^4a_i^2$. This induces a primitive embedding of $\mathbb{Z}(-2d)$ into $A_1(-1)^{\oplus 4}$, 
and one can get the first bound from Lemma~\ref{lemma:irrationality-general}. For the second bound, it follows from \cite[Korollar 1]{HK82} (see also \cite[Section 2]{BBBF23}) that if $d$ is large enough and $d\not\equiv 0,4,7\mod 8$, then it can be expressed as the sum of three coprime squares. In particular, there exists a primitive embedding $\mathbb{Z}\ell\hookrightarrow A_1(-1)^{\oplus 3}$ which leads to the second bound using Lemma~\ref{lemma:irrationality-general}. In the last case, we have a primitive embedding of $\mathbb{Z}(2d)$ into the lattice $Q(-1)$ defined by \eqref{eq:latticeranktwo}, which yields the desired bound from Lemma~\ref{lemma:irrationality-general}.
\end{proof}

\begin{rmk}
When $a=b=1$, the condition of being able to write $d$ in the form $X^2-XY+Y^2$ with $\gcd(X,Y)=1$ is equivalent to the condition $2d\equiv 0,2\mod 6$ and $2d$ not divisible by $4$, $9$ or any odd prime $p\equiv 2\mod 3$. This is the standard necessary and sufficient condition for a polarized K3 surface of degree $2d$ to admit an associated labeled special cubic fourfold; see \cite{Has00} or \cite[Theorem~23]{Has16}.
\end{rmk}


\newcommand{\etalchar}[1]{$^{#1}$}

\end{document}